\documentclass[12pt,a4paper]{article}
\usepackage[utf8x]{inputenc}
\usepackage[english]{babel}
\usepackage[T1]{fontenc}
\usepackage{amsmath}
\usepackage{amsfonts}
\usepackage{amssymb}
\usepackage{makeidx}
\usepackage{graphicx}
\usepackage[left=3cm,right=2cm,top=3cm,bottom=2cm]{geometry}
\usepackage{setspace}
\usepackage{amsthm}
\usepackage{mathtools}
\usepackage{cite}
\usepackage[colorlinks=true,urlcolor=blue,
citecolor=red,linkcolor=blue,linktocpage,pdfpagelabels,
bookmarksnumbered,bookmarksopen]{hyperref}

\newtheorem{theo}{Theorem}
\newtheorem{prop}{Proposition}

\newtheorem{lema}{Lemma}
\newtheorem{aff}{Claim}

\newtheorem{obs}{Remark}

\newcommand{\R}{\mathbb{R}}

\newcommand{\ud}{\mathrm{d}}

\usepackage{times, color, xcolor}

\begin{document}



\title{MULTIPLE NODAL SOLUTIONS OF PLANAR STEIN-WEISS EQUATIONS}

\date{}

\author{
		EUDES M. BARBOZA$^1$ \thanks{Corresponding author}, EDUARDO DE S. BÖER$^{2}$  ,  OLÍMPIO H. MIYAGAKI$^{3}$ \\ \,\, and \,\, CLAUDIA R. SANTANA$^{4}$}
\noindent 

\maketitle

\noindent \textbf{Abstract:} In this paper, our goal is to investigate the existence of multiple nodal solutions to a class of planar Stein-Weiss problems involving a nonlinearity $f$ with subcritical or critical growth in the sense of Trudinger-Moser. To achieve this, we combine a gluing approach with the Nehari manifold argument. We demonstrate that for any positive integer $k\in \mathbb{N}$, the problem studied has at least one radially symmetrical ground state solution that changes sign exactly $k$-times.


\section{Introduction}

In this work, we are concerned with the  following class of planar  equations with Stein-Weiss convolution parts
\begin{equation}\label{problem}
	\left\{\begin{array}{rclcl}\displaystyle -\Delta u+V(|x|) u &=& \dfrac{1}{|x|^{\beta}}\left(\displaystyle\int_{\mathbb{R}^2}\dfrac{F(u(y))}{|y-x|^{\mu}|y|^{\beta}}\,\mathrm{d}y\right)f(u(x))\quad \mbox{in} \quad \mathbb{R}^2;\\
		u \in H^1(\mathbb{R}^2),& & 
	\end{array}\right.
\end{equation} 
with  $\mu>0,\beta\geq0$ such that $\mu+2\beta< 2$, $F(t)= \int_0^tf(s)ds$.  For each $k\in \mathbb{N}$, we seek  a radial symmetric solution for \eqref{problem} with $k$ nodes.

In Problem \eqref{problem}, the nonlinearity is nonlocal and it is motivated by the weighted Hardy-Littlewood-Sobolev inequality that is also known as Stein-Weiss type inequality, see \cite{S-W}. Let us recall this inequality, which is frequently used throughout this paper. 
\begin{lema}\label{WHLS}$($Weighted Hardy-Littlewood-Sobolev inequality$)$ 
	Let $1 < r,s < +\infty$, $0 < \mu < N$, $\gamma + \beta \geq  0$, $0 < \gamma +\beta +\mu \leq N$,
	$g \in L^r(\mathbb{R}^N)$ and $h\in  L^s(\mathbb{R}^N)$. Then, there exists a sharp constant $C(r,s, N,\alpha,\beta,\mu)$ such that
	\begin{equation}\label{hls}
		\int_{\mathbb{R}^N}\!\int_{\mathbb{R}^N}\frac{g(x)h(y)}{|x|^{\gamma}|y-x|^{\mu}|y|^{\beta}}\,\mathrm{d}x\mathrm{d}y\leq C(r,s,N,\gamma,\beta,\mu) \|f\|_r\|h\|_s,
	\end{equation}
	where
	\begin{equation*}
		\frac{1}{{r}}+\frac{1}{s}+\frac{\gamma+\beta+\mu}{N}=2
	\end{equation*}
	and
	\begin{equation*}
		1-\frac{1}{{r}}-\frac{\mu}{N} <\frac{\gamma}{N}<1-\frac{1}{{r}}.
	\end{equation*}
	In addition, for all $h \in L^s(\mathbb{R}^N)$, we have
	\begin{equation}\label{hls1}
		\Bigg\|	\int_{\mathbb{R}^N}\frac{h(y)}{|x|^{\gamma}|y-x|^{\mu}|y|^{\beta}}\,\mathrm{d}y\Bigg\|_t\leq C(t ,s,N,\gamma,\beta,\mu) \|h\|_s,
	\end{equation}
	where $t$ verifies
	\begin{equation}\label{hls1.1}
		1+\frac{1}{t}=\frac{1}{s}+\frac{\gamma+\beta+\mu}{N} \quad \mbox{and} \quad \frac{\gamma}{N}<\frac{1}{t} <\frac{\gamma\mu}{N}.
	\end{equation}
\end{lema}

In 1928, Hardy and Littlewood first introduced the one-dimensional case of \eqref{hls} with $\gamma = \beta = 0$, in \cite{Hardylee}. Sobolev later extended this inequality to the $n$-dimensional version in \cite{Sobolev}, which is now known as the Sobolev-Hardy-Littlewood inequality. Lieb obtained the sharp form of this inequality, along with the existence of a pair of optimal functions, in \cite{Lie83}, using a symmetric decreasing rearrangement technique. Lions provided an alternative proof based on the concentration-compactness principle, in \cite{Lio84}. In 1958, Stein and Weiss generalized the Sobolev-Hardy-Littlewood inequality by introducing two weights, resulting in \eqref{hls}.

The Stein-Weiss inequality plays a crucial role in providing quantitative insights into the integrability of integral operators, closely tied to their dilation characteristics. This inequality has garnered significant interest among scholars due to its fundamental importance in applications related to harmonic analysis and partial differential equations. In recent years, many studies are concened with problems which envolve this kind of convolution term in some different aspects, see for instance, \cite{WS6,SW2,WS3,WS5,WS7,WS1,WS2,WS4} and  references therein.

We reacall that if $\gamma = \beta=0$ \eqref{problem} becomes a Choquard-type equation and  \eqref{hls} reduces to the classical Har{d}y-Littlewood-Sobolev inequality, see \cite{lieb}, which allows to study this kind of equation variationally.
The motivation for the class of Choquard equations  arises from the quest for standing wave solutions in a specific category of time-dependent Schrödinger equations given by

\begin{equation}\label{1}
	i\frac{\partial\Psi}{\partial t}=-\Delta\Psi+W(x)\Psi-\left(I_{\alpha}\ast \vert \Psi \vert^{p} \right)\vert \Psi \vert^{p-2}\Psi, \quad (t,x)\in \mathbb{R}_{+}\times\mathbb{R}^{N},
\end{equation}
where $i$ is the imaginary unit, $p\geq 2$, and $W(x)$ represents an external potential. A standing wave solution of \eqref{1} takes the form $\Psi(x,t)=u(x)e^{-i\omega t}$, where $\omega\in\mathbb{R}$ and $u$ satisfies the stationary equation

\begin{equation}\label{4}
	-\Delta u +V(x)u=\left( I_{\alpha}\ast \vert u\vert^{p}\right)\vert u\vert^{p-2}u, \quad \text{in } \mathbb{R}^N,
\end{equation}
with $V(x)=W(x)-\omega$. In certain instances, \eqref{4} is also recognized as the Schrödinger-Newton equation, initially proposed by R. Penrose, in \cite{penrose} to explore the self-gravitational collapse of a quantum mechanical wave function. When $p>1$, $\alpha=1$, and $N=3$, equation \eqref{problem} takes the form of the nonlinear Choquard equation

\begin{equation}\label{choqu}
	-\Delta u+u=\left(I_{2}\ast |u|^{p}\right)|u|^{p-2}u, \quad x\in\mathbb{R}^{3}.
\end{equation}
This specific case dates back to 1954, as described by S. Pekar, in \cite{pekar}, in which he introduced a polaron at rest within quantum theory. In 1976, P. Choquard used equation \eqref{choqu} to model an electron trapped in its own hole, considering it as an approximation to the Hartree-Fock theory of a one-component plasma, as discussed in \cite{choquard}. For additional insights into the physical background, readers are encouraged to explore \cite{choq1,choq2}.

Returning to case that $\beta,\gamma\geq0$ in \eqref{hls}, notice that if
$g=h=F$, $\gamma =\beta$, $s = r$, then in view of Lemma \ref{WHLS}, we obtain
\begin{equation}\label{Ls1}
	\int_{\mathbb{R}^N} \! \int_{\mathbb{R}^N}\frac{F(u(x))F(u(y))}{|x|^{\beta}|y-x|^{\mu}|y|^{\beta}}\,\mathrm{d}x\mathrm{d}y \leq C(N,\beta,\mu)\|F(u)\|^2_{s},
\end{equation}
where $s>1$ is defined by 
\begin{equation*}
	\frac{2}{s}+\frac{2\beta+\mu}{N}=2,
\end{equation*}
i.e.,
\begin{equation*}
	s=\frac{2N}{2N-2\beta -\mu}.
\end{equation*}
This means we must require
\begin{equation*}
	F(u)\in L^{\frac{2N}{2N-2\beta -\mu}}(\mathbb{R}^N).
\end{equation*}
By considering $N\geq 3$ and the pure power $F(t)=|t|^{q}$, we may use Sobolev embedding when
\begin{equation*}
	\frac{2Nq}{2N-2\beta-\mu} \in \left[2,2^*\right],
\end{equation*}
i.e., when the exponent $q$ satisfies
\begin{equation*}
	2_{^*\beta,\mu}:=\frac{2N-2\beta-\mu}{N} \leq q\leq 	\frac{2N-2\beta-\mu}{N-2}=: 2^*_{\beta,\mu}.
\end{equation*}
The lower bound $2_{^*\beta,\mu}$ is called \textit{lower critical exponent} by Du-Gao-Yang, in \cite{WS7}, and the authors also called $2^*_{\beta,\mu}$ the \textit{upper critical exponent} in the sense of the weighted Har{d}y-Littlewood-Sobolev inequality. Roughly speaking, for a bounded domain $\Omega\subset\mathbb{R}^2$,  we obtain this kind of embedding when $q\in \left[\dfrac{4-2\beta +\mu}{2}, \infty\right]$, which do not occur in $L^{\infty}(\Omega)$. Hence, to overcome the obstacle in the limiting case, the Trudinger-Moser
inequality can be treated as a substitute of the Sobolev inequality (see \cite{Moser,Trudinger}).  This inequality can be summarized as follows
\begin{equation}\label{T-M}
	\displaystyle\sup_{\substack{u\in H^1_0(\Omega)\\ \|\nabla u\|_2=1}} \int_{\Omega} e^{\beta u^2} \; \mathrm{d} x\left\{\begin{array}{lcr}
		C|\Omega|& \mbox{if} & \beta \leq 4\pi; \\
		=+\infty & \mbox{if} &  \beta> 4\pi.
	\end{array}
	\right.
\end{equation}
Therefore, from \eqref{T-M}, we have naturally associated notions of subcriticality and criticality (see \cite{Adimurthi,Miyagaki}): we say that
\begin{description}
	\item[$(SG)$] $f(t)$ has subcritical growth  if for all $\alpha$ one has
	\begin{equation*}
		\displaystyle \lim_{|t|\rightarrow +\infty} \frac{|f(t)|}{e^{\alpha t^2}}=0;
	\end{equation*}
	
	\item[$(CG)$]
	$f(t)$ has critical growth 
	\begin{equation*}
		\begin{array}{lrlr}\displaystyle\lim_{|t|\rightarrow +\infty} \frac{|f(t)|}{e^{\alpha t^2}}=0,& \forall \alpha>4\pi;&\displaystyle\lim_{t\rightarrow +\infty} \frac{|f(t)|}{e^{\beta t^2}}= +\infty,& \forall \alpha<4\pi.
		\end{array}
\end{equation*}\end{description}

This definition was introduced by  \cite{Adimurthi}; see also  \cite{Miyagaki} for further reference. However, the supremum in \eqref{T-M} becomes infinite for domains with infinite measure, rendering this inequality inapplicable to unbounded domains. In our context, we will present the following version of the Trudinger-Moser inequality on the entire $\mathbb{R}^2$, as found in \cite{Cao,doO}.
\begin{lema}
	If $\alpha>0$ and $u \in H^1(\mathbb{R}^2)$, then
	\begin{equation}\label{exponencialfinita}\int_{\mathbb{R}^2} (e^{\alpha u^2}- 1) \mathrm{d} x< \infty.\end{equation}
	Moreover,
	if $\|\nabla u\|_2\leq1$, $\|u\|_2\leq M<\infty$, and $\alpha<4\pi$, then there exists a constant $C$, which depends only on $M$ and $\alpha$, such that
	
	\begin{equation}\label{TM}\sup_{\substack{u\in H^1(\mathbb{R}^2)\\ \|\nabla u\|_2=1}} \int_{\mathbb{R}^2} (e^{\alpha u^2}- 1) \mathrm{d} x< C(\alpha,M).\end{equation}
\end{lema}

We highlight that nowadays attention is being given to problems involving the Stein-Weiss convolution term with a nonlinearity exhibiting exponential growth. In \cite{WS6}, the authors study planar Schrödinger equations with Stein-Weiss convolution parts and a nonlinearity that exhibits critical exponential growth in the Trudinger-Moser sense. They leveraged the general minimax principle to investigate the existence of mountain-pass type solutions for the equation and demonstrated that the mountain-pass value equals the least energy level. In \cite{WS8}, it was proven that normalized solutions exist for a nonlinear Schrödinger equation with the Stein-Weiss reaction under the exponential critical growth case in $\mathbb{R}^2$. In \cite{WS2}, positive solutions to a class of quasilinear Schrödinger equations involving Stein-Weiss type convolution and exponential nonlinearity in $\mathbb{R}^N$ were investigated. In \cite{WS9}, the existence of solutions was demonstrated for a class of quasilinear Schrödinger equations introduced in plasma physics and nonlinear optics with Stein-Weiss convolution parts and a critical exponential growth nonlinearity in $\mathbb{R}^2$. We also refer readers to \cite{alvesetal,alvesMimbo,deng} for studies on Choquard-type equations with exponential nonlinearities in $\mathbb{R}^2$.

However, as far as we know, there has not been any work that addresses problems with Stein-Weiss convolution term combined with a nonlinearity involving exponential growth concerning nodal solutions. Based on this observation, our main goal here is to demonstrate that for any $k\in\mathbb{N}$, there exists a solution to \eqref{problem} with exactly $k$ nodes. For this matter, we will use a gluing approuach in the same spirit of \cite{GuiGuo,nodal}.

This kind of argument was developed in \cite{primeiro,segundo}. They first obtained the least energy solutions of a problem without the convolution term in each annulus, including every ball and its complement, and then glued them by matching the normal derivative at each junction point. However, this method cannot be applied directly to an equation with nonlocal terms, because we cannot solve this kind of term separately on each annulus once solutions require global information. To address this, we use a technique developed by Y. Deng, S. Peng, and W. Shuai in \cite{Chern} for a problem with a nonlocal Chern–Simons term. They overcame this difficulty by treating their problem as a system of 
$(k+1)$ equations with 
$(k+1)$ unknown functions, each supported on a single annulus and vanishing outside it. This method was adapted for a Choquard-type equation with subcritical Sobolev-Hardy-Littlewood growth in \cite{GuiGuo, nodal}.  

Here, our  main challenge lies in applying a gluing approach together with the Nehari
manifold argument and a problem involving exponential nonlinearity. Summarizing, initially, we divide $\mathbb{R}^2$ into $k+1$ symmetric distinct parts (one ball, $k-1$ annuli, and one ball complement). After that, we define  $k+1$ suitables Hilbert spaces of functions, each one with domains in these respective parts of $\mathbb{R}^2$, and introduce an auxiliary functional in the cartesian product of these spaces. So, using a Nehari manifold that depends on $k$, we seek critical points of this  functional with $k+1$ components $u_i$, for $i=1,\cdots , k+1$, each  with a definite sign. As we are working with an exponential nonlinearity, the Trudinger-Moser inequality is essential to obtain suitable estimates. In the subcritical case, as we can take $\alpha>0$ in \eqref{TM}, we can derive some estimates quickly. However, in the critical case, we need to establish an appropriate minimax estimate of this functional to achieve similar results as in the previous case. Finally, by gluing $u_1,\cdots, u_{k+1}$, and adjusting their signs as necessary, we define $u=\Sigma_{i=1}^{k+1}u_i$. Then, by the Brouwer degree, we prove that $u$ is indeed a solution for \eqref{problem}.

In this paper, the novelty lies in applying this approach to a problem that combines an exponential nonlinearity with a convolution term using Stein-Weiss weights. Consequently, we encounter new difficulties since we deal with a general nonlinearity that initially satisfies mild properties. However, to utilize this method effectively, we need to calculate second-order derivatives of the functional associated to an auxiliary problem, which plays an important role. Therefore, it is natural to assume certain behavior under the derivative of the nonlinearity. Thus, we make non-standard assumptions regarding $f$ (see $(f_3)$). We emphasize that although these new hypotheses are not very common, standard examples of nonlinearity with exponential growth satisfy them. For further details, please refer to items (a) and (c) of Remark \ref{obs1} and Remark \ref{obs3}.

Finally, we point out that our contributions complement and extend some results from \cite{alvesMimbo,GuiGuo} in some ways. Specifically, we have derived both subcritical and critical exponential versions of \cite[Theorem 1.1-(i)]{GuiGuo} when 
$ N=2$. Additionally, we have obtained nodal solutions for a problem similar to the one addressed in \cite{alvesetal}. In both cases, our results align with those in \cite{alvesMimbo, GuiGuo} when $\beta=0$.

\subsection{Hypotheses and main results}
We will present some hypotheses to study Problem \eqref{problem} variationally. Initially, we assume

\begin{description}
	\item[(V)] $V\in C([0,\infty], \mathbb{R}^+)$ is radial and bounded below a positive
	constant.\end{description}
Denote $H_{\mathrm{rad}}^{1}(\mathbb{R}^N)$ the radial function subspace of $H^1(\mathbb{R}^N)$ and
\[H_V := \left\{u\in H_{\mathrm{rad}}^{1}(\mathbb{R}^N);  \int_{\mathbb{R}^N}V(|x|)u^2\, \mathrm{d}x< \infty\right\},\]
by {\bf (V)}, we have that $H_V$ is a Hilbert space with the inner product
\[\langle u,v\rangle = \int_{\mathbb{R}^2}\nabla u\nabla v \mathrm{d}x+ \int_{\mathbb{R}^2} V(x)uv \mathrm{d}x \] and the norm
\[\|u\|= \int_{\mathbb{R}^2}|\nabla u|^2 \mathrm{d}x+ \int_{\mathbb{R}^2} V(x)u^2. \mathrm{d}x \]
Before stating our main results, we shall introduce the following assumptions on the nonlinearity:
\begin{itemize}
	\item[$(f_1)$]$f \in {C^1}(\R,\R^+)$ and $f(s)=0$ if $s\leq 0$;
	\item[$(f_2)$] $f(s)= o(s^{\frac{3[4 - (\mu + 2\beta)]}{2}-1})$, as $s \rightarrow 0;$
	\item[$(f_3)$]there exist  ${\theta} >2$ such that $f'(s)s^2>f(s)s> \theta F(s)>0$ for all $s>0$.

\end{itemize}
The main results of this paper are presented below. \begin{theo}
	\label{teo1} Assume  {\bf (V)}, $(SG)$, $(f_1)-(f_3)$, and  
	\begin{itemize}
		
		\item[$(f_4)$]there exist $s_0$ and $M>0$ such that
		\begin{equation*} F(s)=\displaystyle\int_0^sf(t)\; \ud t\leq M f(s) \mbox{ for all } s>s_0.\end{equation*}
		
	\end{itemize} Then, for each $k\in \mathbb{N}$, Problem \eqref{problem} has a radial symmetric ground state solution $u_k$ that changes sign $k$ times. 
\end{theo}

\begin{theo}
	\label{teo2} Assume  {\bf (V)}, $(CG)$, $(f_1)-(f_3)$, and  
	\begin{itemize}

		\item[$(f_5)$] there are constants $p>4 -(\mu+2\beta)$ and $C_{p}>0$ such that
		\[
		F(s)\geq C_{p}s^{p}, \quad\text{for all}\quad s\geq 0,\]

	\end{itemize}	with $C_p$ large enough. Then, for each $k\in \mathbb{N}$, Problem \eqref{problem} has a radial symmetric ground state solution $u_k$ that changes sign $k$ times. 
\end{theo}

\begin{obs}\label{obs1} Here we present some details regarding the hypotheses of the previous theorems.
	\begin{itemize}

		\item[(a)] The assunption $(f_3)$ include the well-known Ambrosetti--Rabinowitz condition, in exact terms,  there exist $\theta>2$ such that
		\begin{equation}\label{h2}\theta F(s) \leq  f(s)s \quad\mbox{for all}\quad  s>0.
		\end{equation}
		The Ambrosetti-Rabinowitz condition has been used in most studies and plays an important role in investigating the existence of nontrivial solutions for many nonlinear elliptic problems. Here, we use \eqref{h2} to ensure the uniqueness of a global point of an auxiliary functional on a specific set $\mathcal{\tilde{N}}$. However, this condition is not sufficient to guarantee this result; in this case, we need more. Specifically, we use the fact that  $f'(s)>0$. Moreover, to prove that $\mathcal{\tilde{N}}$ is a kind of Nehari manifold, we also need 
		$f'(s)s^2>f(s)$ ( for detail, see Lemmas \ref{multipleneharivetor} and \ref{manifold}). For this matter, we assume $(f_3)$ as in \cite{Clapp}, where the authors also address nodal solutions, but, for problems with nonlinearities exhibiting Sobolev growth.
		
		\item[(b)] We can see that from conditions $(f_1)$ and $(f_4)$, it follows that there are constants $C, s_0>0$ such that
		\begin{equation*}\label{exponen}
			\displaystyle F(s)>C e^{{|s|}/{M}}\quad\mbox{for all}\quad s>s_0.
		\end{equation*}
		Hence we can see that $F(s)$ and $f(s)$ have exponential growth for $t$ large enough.
		Consequently, as $f$ is continous, there exist $C>0$ and $p>2$ such that  
		\begin{equation}\label{estfrombelow}F(s) >C|s|^{p} \quad {\mbox{for }\quad s>s_0}.\end{equation}
		\item[(c)]Except for assuming  $f'(s)s^2> f(s)s$, conditions $(f_1)-(f_5)$ are the standard assumptions for treating a problem with exponential nonlinearity variationally. However, models for 
		$f$ that satisfy these usual hypotheses naturally meet the condition $f'(s)s^2> f(s)s$. For instance,  
		$$f(s)=  s^{\frac{3[4 - (\mu + 2\beta)]}{2}-1}+ e^s$$ is a model of nonlinearity for Theorem \ref{teo1}, and	$$f(t)=  s^{\frac{3[4 - (\mu + 2\beta)]}{2}-1}+ e^{4\pi s^2},$$  for Theorem \ref{teo2}.
		
		\item[(d)] In our context, we assume $(f_4)$ in order to obtain \eqref{estfrombelow}, which helps us to show a certain behavior of functional associated to Problem \eqref{problem} at infnity. Notice that $(f_5)$ guarantees  a similar estimate to \eqref{estfrombelow}, but requering an apropriated constant $C_p$ that should be large enough.    Hypotheses as $(f_5)$  were introduced by  \cite{Cao} and are used to prove that the minimax level of a functional associated to a problem involving a nonlinearity with  critical growth in the sense of Trudiger-Moser. However, in Theorem \ref{teo2}, $(f_5)$ also be used in the place of $(f_4)$, in view of \eqref{estfrombelow}.
	\end{itemize}
\end{obs}
\subsection{Outline}
This paper is organized as follows. In Sect.\ref{prel}, we introduce the tools needed to apply the gluing method. In Sect. \ref{subcriticalcase}, we consider a nonlinearity with subcritical growth and prove Theorem \ref{teo1}. In Sect. \ref{criticalcase}, we study Problem \eqref{problem} in the critical growth range. To achieve this, we demonstrate the boundedness of the minimax levels of the functional associated with an auxiliary problem and then prove Theorem \ref{teo2}.

\section{Preliminaries}\label{prel}
Here, we are concerned with Problem \eqref{problem}. Notice that the  energy functional associated to problem \eqref{problem} is $J: H_V\rightarrow \mathbb{R}$, given by
\begin{equation*}\label{fuctionalJ}
	J(u)= \dfrac{1}{2}\int_{\mathbb{R}^2}\left(|\nabla u|^2 + V(|x|)u^2\right) \, \mathrm{d} x -  \frac{1}{2}\int_{\mathbb{R}^2}\int_{\mathbb{R}^2}\displaystyle\dfrac{F(u)F(u)}{|x|^{\beta}|y-x|^{\mu}|y|^{\beta}}\, \mathrm{d} x\mathrm{d} y .		
\end{equation*} 
By the Weighted Hardy-Littlewood-Sobolev  and Trudiger Moser inequalities $J\in C^1(H_V,\mathbb{R})$, and by the principle of symmetric criticality, the critical points of $J$ in $H_V$ are weak solution of \eqref{problem}.

Now, we will present the elements necessary to apply the method for obtaining nodal solutions based on \cite{GuiGuo, nodal}. We set
\[c_0 =\inf_{u\in\mathcal{N}} J(u),\]
where
\[\mathcal{N}=\left\{u \in H_V\setminus\{0\}; J'(u)u=0\right\}\]
is the Nehari manifold. In order to state ours results conveniently, for each positive integer $k$, we denote ${R_k}=(r_1,..., r_k) \in \mathbb{R}^k$, with $0<r_1<r_2< \cdots < r_k< r_{k+1}= \infty$, define a Nehari type set 
\begin{equation*}\label{Neharik}
	\mathcal{N}_k= \left\{ \begin{array}{c}u\in H_V;\mbox{ there exists } R_k \mbox{ such that } u_i\neq 0 \mbox{ in } B_i,\\ u_i =0 \mbox{ on } \partial B_i, \mbox{ and } J'(u)u_i=0 \mbox{ for all } 1\leq i \leq k+1\end{array}\right\},
\end{equation*} 
where $ B_i= \{x\in \mathbb{R}^2; r_{i-1}\leq |x|\leq r_{i}\},$ with $r_0=0,  u_i=u$ in $B_i$ and $u_i=0$ outside of $B_i$. Notice that $u=\sum_{i=1}^{k+1}u_i$.
Then, we consider the infimmum level
\begin{equation}\label{ck}
	c_k= \inf_{u\in \mathcal{N}_k} J(u).
\end{equation}

Throughout this paper, we denote by $C$ positive constants. For positives integer $k$, we define
\[\Gamma_k= \{ R_k = (r_1,\cdots, r_k)\in \mathbb{R}^k; 0< r_1<r_2< \cdots < r_k< r_{k+1}= \infty\},\]
and for each $R_k\in \Gamma_k$, we denote
\[B_1 = B^{R_k}_1=\left\{ x\in \mathbb{R}^2; 0\leq |x|\leq r_1\right\}, B_i = B^{R_k}_i= \{x\in \mathbb{R}^2; r_{i-1}\leq |x|\leq r_{i}\}\] for $ i= 2,\cdots,k+1$
and
\[ H_i =  H_i^{R_k}= \{v\in H_0^1(B_i^{R_k}); v(x) = v(|x|), v(x) \equiv 0 \mbox{ if } x\not \in B_i^{R_k}\} \quad \mbox{and}\]
\[\quad \overline{H}_k = \mathcal{H}_k^{R_k}=  H_1^{R_k}\times \cdots \times  H_{k+1}^{R_k} = H_1\times \cdots \times  H_{k+1}.\]

Notice that by {\bf (V)}, $ H_i$  is a Hilbert space with norm
$$\|v\|_i = \left(\int_{B_i}(|\nabla v|^2 + V(|x|)v^2 \, \mathrm{d} x \right)^{1/2}.$$
By Strauss Lemma, we also have that these embeddings
\begin{equation}\label{embedding}
	H_i \hookrightarrow H^1_{rad} (B_i)\hookrightarrow L^q(B_i)\end{equation}are compact for $2 < q < \infty$ and all $i=1,\cdots,k+1$.

Now, we introduce a new functional $E:\mathcal{H}_k \rightarrow \mathbb{R}$ defined by
\[\begin{array}{rl}
	{E}(u_1,\cdots, u_{k+1})=& \displaystyle\sum_{i=1}^{k+1} \left(\dfrac{1}{2}\|u_i\|_i^2- \dfrac{1}{2}\int_{B_i}\int_{B_i}\displaystyle\dfrac{F(u_i)F(u_i)}{|x|^{\beta}|y-x|^{\mu}|y|^{\beta}}\, \mathrm{d} x\mathrm{d} y\right) \\&\displaystyle - \dfrac{1}{2}\sum_{\substack{i,j=1\\ i\neq i}}^{k+1}  \int_{B_i}\int_{B_j}\displaystyle\dfrac{F(u_i)F(u_j)}{|x|^{\beta}|y-x|^{\mu}|y|^{\beta}}\, \mathrm{d} x\mathrm{d} y
\end{array}\] 
which is related to the functional $J$, because
\begin{equation}\label{JE2}
	{E}(u_1, \cdots, u_{k+1})= J\left(\sum_{i=1}^{k+1}u_i\right).
\end{equation}
Note that
\[ \begin{array}{rl}
	\partial_{u_i}{E}(u_1,\cdots, u_{k+1})u_i=& \displaystyle \left(\|u_i\|_i^2 - \displaystyle\int_{B_i}\int_{B_i}\dfrac{F(u_i)f(u_i)u_i}{|x|^{\beta}|y-x|^{\mu}|y|^{\beta}}\,\, \mathrm{d} x\mathrm{d} y\right) \\&\displaystyle - \frac{1}{2
		
	}\sum_{\substack{ i\neq j}}\int_{B_i}\int_{B_j}\dfrac{F(u_j)f(u_i)u_i}{|x|^{\beta}|y-x|^{\mu}|y|^{\beta}}\,\, \mathrm{d} x\mathrm{d} y .
\end{array} \]
Now, we set the Nehari type set
\[\mathcal{N}^{R_k}_k= \{(u_1,\cdots ,u_k)\in \mathcal{H}^{R_k}_k; u_i\neq 0, \partial_{u_i}{E}(u_1, \cdots, u_{k+1})u_i=0, i=1,\cdots , k+1\}\]
and de minimum value
\begin{equation}
	\label{dk2}
	d(R_k): = \inf_{(u_1, \cdots, u_{k+1})\in \mathcal{N}^{R_k}_k} E(u_1, \cdots, u_{k+1}).
\end{equation}
By \eqref{ck}, this, together with \eqref{JE2}, we have
\begin{equation}\label{jd2}
	\inf_{u\in \mathcal{N}_k}J(u)= c_k = \inf_{R_k \in \Gamma_k} d(R_k).
\end{equation}

Notice that if $ (u_1,\cdots, u_k) \in \mathcal{H}_k$ is a critical point of $E$, then each component $u_i$ satisfies
\begin{equation}\label{probi}
	\left\{\begin{array}{rclcl}\displaystyle -\Delta u_i+V(|x|) u_i &=& \displaystyle\sum_{j=1}^k\int_{B_i}\dfrac{F(u_j)}{|x|^{\beta}|y-x|^{\mu}|y|^{\beta}}f(u_i)u_i\,\mathrm{d} y\quad \mbox{in} \quad B_i;\\
		u_i =0 \mbox{ in } (B_i)^c.& &
	\end{array}\right.
\end{equation} 

Now, we remark some estimate of $f$ and $F$, which will helps us throughout this paper. By using assumptions $(f_1)-(f_2)$ and $(SG)$ or $(CG)$, it follows that for each  $q>2$ and $\varepsilon>0$ there exists $C_{\varepsilon}>0$ such that 
\begin{equation}\label{growth}
	f(t)\leq \varepsilon |t|^{\frac{3[4 - (\mu + 2\beta)]}{2}-1}+C_{\varepsilon}(e^{\alpha t^{2}}-1)|t|^{q-1}, \quad \mbox{for all} \hspace{0,2cm} t\in\mathbb{R},
\end{equation} 
which implies that
\begin{equation}\label{growth2}
	F(t)\leq \tilde{\varepsilon}|t|^{\frac{3[4 - (\mu + 2\beta)]}{2}}+\tilde{C_{\varepsilon}}(e^{\alpha t^{2}}-1)|t|^{q}, \quad \mbox{for all} \hspace{0,2cm} t\in\mathbb{R}.
\end{equation}

\begin{obs}
	It is important to stress that if we assume $(SG)$, \eqref{growth} and\eqref{growth2} are avaliable for any $\alpha>0$. If we suppose $(CG)$, we are considering $\alpha\geq4\pi$.  
\end{obs}
From \eqref{growth2}, we have
\begin{equation}\label{estimativeF}
	\left\|F(u) \right\|_{\frac{4}{4-(\mu+2\beta)}}
	\leq \varepsilon\|u\|_{6}^{\frac{3[4 - (\mu + 2\beta)]}{2}}+C\left\{\int_{\mathbb{R}^2}\left[\left(e^{\alpha u^{2}}-1\right)|u|^{q}\right]^{\frac{4}{4-(\mu+2\beta)}}\,\mathrm{d}x\right\}^{\frac{4-(\mu+ 2\beta)}{4}}.
\end{equation} 
Consider $\varrho>0$ and suppose $\|u\|\leq \varrho$. By H\"older inequality we obtain
\begin{equation}\label{rj2}
	\int_{\mathbb{R}^2}\left[\left(e^{\alpha u^{2}}-1\right)|u|^{q}\right]^{\frac{4}{4-(\mu+2\beta)}}\,\mathrm{d}x\leq \left[\int_{\mathbb{R}^2}\left(e^{\frac{8\alpha \|u\|^{2}}{4-(\mu+2\beta)}\left(\frac{u}{\|u\|}\right)^{2}}-1 \right)\,\mathrm{d}x \right]^{\frac{1}{2}}\|u\|_{\frac{8q}{4-(\mu+2\beta)}}^{\frac{4q}{4-(\mu+2\beta)}}.
\end{equation}
If $\varrho\leq \sqrt{4\pi(4-(\mu+2\beta))/8\alpha}$, then we are able to apply Trudinger-Moser inequality. Thus, \eqref{estimativeF}, \eqref{rj2} jointly with Sobolev embedding imply that
\[
\|F(u)\|_{\frac{4}{4-(\mu-2\beta)}}\leq \varepsilon\|u\|^{\frac{3(4-(\mu-2\beta))}{2}}+\tilde{C}\|u\|^{{q}}.
\]
Hence, for all $i,j=1,\cdots , k=1$, it follows from Weighted Hardy-Littlewood-Sobolev  inequality that
\begin{equation}
	\label{estimativeF2}
	\int_{B_i}\int_{B_j}\dfrac{F(u_i)F(u_j)}{|x|^{\beta}|x-y|^{\mu}|y|^{\beta}}\, \mathrm{d} x\mathrm{d} y \leq \varepsilon^2(\|u_i\|_i\|u_j\|_j)^{{\frac{3[4 - (\mu + 2\beta)]}{2}}}+\tilde{C}(\|u_i\|_i\|u_j\|_j)^{{q}}.
\end{equation}
\section{Subcritical case}\label{subcriticalcase}
This section is devoted to the proof of Theorem \ref{teo1}. Here, we assume the nonlinearity has subcritical exponential growth and conditions $(f_1)-(f_4)$. In this context, some proofs are simpler than similar results when considering $(CG)$, as we can take $\alpha > 0$ small enough in \eqref{growth} and \eqref{growth2}, which is not possible in the critical case. 

\subsection{Critical points of the functional $E$ in a Nehari type manifold }\label{CritalE}
Here, we will show some results about critical points of $E$ in $\mathcal{N}_k$.
\begin{lema}\label{multipleneharivetor} Assume $(SG)$, $(f_1)-(f_4)$. Let $q \in \left(2, \dfrac{\theta}{2}+1\right]$, where $\theta$ is given in $(f_3)$. For $(u_1, \dots , u_{k+1}) \in  \mathcal{H}_k$ with $u_i\neq 0$
	for $i = 1, \cdots , k + 1$, there is a unique $(k + 1)$-{tuple} $(t_1, \cdots , t_{k+1})$ of positive
	numbers such that $(t_1^{1/q}u_1, \cdots, t_{k+1}^{1/q}u_{k+1}) \in \mathcal{N}^{R_k}_k$.
	Moreover,  $$(t_1, \cdots , t_{k+1}) \in  (\mathbb{R}>0)^{k+1}:= \{(x_1,\cdots,x_{k+1}) \in \mathbb{R}^{k+1}; x_i>0 \quad\mbox{for} \quad i=1,\cdots, k+1\}$$ is the unique global maximum point of $G:(\mathbb{R}>0)^{k+1} \rightarrow \mathbb{R}$ defined by \linebreak $G(t_1, \dots , t_{k+1}) := E(t_1^{1/q}u_1,\dots ,t_{k+1}^{1/q}u_{k+1})$.
\end{lema}
\begin{proof}
	First, notice that
	\[\begin{aligned}
		G(t_1, \dots , t_{k+1}) :&= E(t_1^{1/q}u_1,\dots ,t_{k+1}^{1/q}u_{k+1})\\
		&= \displaystyle\sum_{i=1}^{k+1} \dfrac{t_i^{2/q}}{2}\|u_i\|_i^2 - \dfrac{1}{2}\sum_{i=1}^{k+1} \sum_{j=1}^{k+1}\int_{B_i}\int_{B_j}\dfrac{F(t_i^{\frac{1}{q}}u_i)F(t_j^{\frac{1}{q}}u_j)}{|x|^{\beta}|x-y|^{\mu}|y|^{\beta}}\, \mathrm{d} x\mathrm{d} y. \,\end{aligned}\]
	By $(f_1)$,  $G(t_1,\dots , t_{k+1})  \rightarrow  0$ as $|(t_1,\dots , t_{k+1})| \rightarrow 0$.	Now, take $(u^0_1, \cdots, u^0_{k+1})\in \overline{H}_{k}\setminus\{0\}$, $u^0_{i}\geq0$ for all $i=1,\cdots, k+1$. As $F(t)\geq 0$, we get 
	\[G(t_1, \dots , t_{k+1})\leq \displaystyle\sum_{i=1}^{k+1} \dfrac{t_i^{2/q}}{2}\|u_i^0\|_i^2 - \displaystyle\sum_{i=1}^{k+1}\int_{B_i}\int_{B_i}\dfrac{F(t_i^{\frac{1}{q}}u^0_i)F(t_i^{\frac{1}{q}}u^0_i)}{|x|^{\beta}|x-y|^{\mu}|y|^{\beta}}\, \mathrm{d} x\mathrm{d} y \,. \]					
	Morerover, by \eqref{estfrombelow},  for $t$ large enough, we have
	\[G(t_1, \dots , t_{k+1})\leq \displaystyle\sum_{i=1}^{k+1} \dfrac{t_i^{2/q}}{2}\|u_i^0\|_i^2 - \displaystyle \dfrac{Ct_i^{\frac{2p}{q}}}{2}\sum_{i=1}^{k+1}\int_{B_i}\int_{B_i}\dfrac{|u^0_i|^{p}|u^0_j|^p}{|x|^{\beta}|x-y|^{\mu}|y|^{\beta}}\, \mathrm{d} x\mathrm{d} y \,. \]	
	As $p>2$,
	$G(t_1,\dots , t_{k+1}) \rightarrow - \infty$ as $|(t_1,\dots , t_{k+1})| \rightarrow + \infty$. Then $G$ possesses a global maximum point $(\bar{t}_1, \dots, \bar{t}_{k+1}) \in  (\mathbb{R}>0)^{k+1}$. Furthermore, this global maximum point is unique. In fact, a direct calculation gives that
	\[\begin{aligned}
		\left ( \dfrac{\partial^{2}G}{
			\partial t_i \partial t_j}
		\right )_{(k+1)×(k+1)} = (a_{ij})_{(k+1)×(k+1)} + (b_{ij})_{(k+1)×(k+1)} 
	\end{aligned},\]
	where
	\[ \scriptsize	a_{ij} = \left\{ \begin{array}{lll}
		\displaystyle\dfrac{(2-q)}{q^2}t_i^{\frac{2}{q}-2}\|u_i\|_i^2 \	{-\frac{1}{q^2t^2_i} \int_{B_i}\int_{B_i}\dfrac{f'(t_i^{\frac{1}{q}}u_i(x)t_i^{\frac{2}{q}}u^2_i(x)F(t_i^\frac{1}{q}u_i(y))}{|x|^{\beta}|x-y|^{\mu}|y|^{\beta}}\, \mathrm{d} x\mathrm{d} y}\\\\
		\displaystyle-\frac{1}{q^2t^2_i} \int_{B_i}\int_{B_i}\dfrac{\left[\left(1-q\right) F(t_i^{\frac{1}{q}}u_i(y)) + \dfrac{f(t_i^{\frac{1}{q}}u_i(y))t_i^{\frac{1}{q}}u_i(y)}{2}\right] f(t_i^{\frac{1}{q}}u_i(x))t_i^{\frac{1}{q}}u_i(x) }{|x|^{\beta}|x-y|^{\mu}|y|^{\beta}}\, \mathrm{d} x\mathrm{d} y&\\\\\displaystyle
		- \dfrac{1}{2q^2t_i^2}\sum^{k+1}_{\substack{l=1\\ l\neq i}}\int_{B_i}\int_{B_l}\dfrac{F(t_l^{\frac{1}{q}}u_i(y)) f'(t_i^{\frac{1}{q}}u_i(x))t_i^{\frac{2}{q}}u^2_i(x) }{|x|^{\beta}|x-y|^{\mu}|y|^{\beta}}\, \mathrm{d} x\mathrm{d} y, \quad \mbox{if}\quad i = j; \\
		0,\quad \mbox{ if }\quad i \neq j
	\end{array}\right.\]

	and
	\[ \scriptsize	b_{ij}=\left\{\begin{aligned}
		-\frac{1}{2t_it_j} \int_{B_i}\int_{B_j}\dfrac{f(t_i^{\frac{1}{q}}u_i(x))t_i^{\frac{1}{q}}u_i(x)f(t_j^{\frac{1}{q}}u_j(y))t_j^{\frac{1}{q}}u_j(y)}{|x|^{\beta}|x-y|^{\mu}|y|^{\beta}}\, \mathrm{d} x\mathrm{d} y \,\, \text{for all}\, i, j = 1, \dots, k+1.\\
	\end{aligned}\right.\]
	Notice that for all $i=1,\cdots,k+1$, $a_{ii}<0$, because $2<q\leq\dfrac{\theta}{2}+1$, and, $F(t),f(t),f'(t)$ are positive for $t>0$, due to  $(f_3)$. By $(f_3)$, we also conclude that $b_{i,j}<0$ for all $i,j=1,\cdots,k+1$ and $t>0$. Thus, from \cite[Lemma A.1]{GuiGuo}, it follows that $(b_{ij})_{(k+1)×(k+1)}$ is negative definite if we set $v_i = \dfrac{f(t_i^{\frac{1}{q}}u_i)t_i^{\frac{1}{q}}u_i}{\sqrt{2}t_i}$. So, the Hessian matrix
	\[\begin{aligned}
		\left ( \dfrac{\partial^{2}G}{\partial t_i \partial t_j}
		\right )_{(k+1)×(k+1)}  \in (\mathbb{R}>0)^{k+1}
	\end{aligned} \]  is negative definite, and the function $G$ is strictly concave, which implies the uniqueness once $f\in C^1(\mathbb{R}), \mathbb{R}^+$. Thus, $(\bar{t}_1, \dots, \bar{t}_{k+1})$   is the unique global maximum point of $G \,\, \text{in}\, \, {(\mathbb{R}>0)^{k+1}}$.
This implies $\partial_{u_i} E(\bar{t}_1^{1/q}u_1,\dots ,\bar{t}_{k+1}^{1/q}u_{k+1})u_i = 0 $ and then $(\bar{t}_1^{1/q}u_1,\dots ,\bar{t}_{k+1}^{1/q})  \in \mathcal{N}^{R_k}_k$. Hence, this lemma follows.
\end{proof}
\begin{obs}\label{obs3}
	We highlight that in Lemma \ref{multipleneharivetor}, it is not necessary to assume that $f'(s)s^2 > f(s)s$. To prove this result, the assumption is that the derivative of the nonlinearity is positive.
\end{obs}
\begin{lema}\label{manifold} Assume  $(f_1)-(f_3)$, and  $(SG)$, then $\mathcal{N}^{R_k}_k$ is a differentiable manifold in $\mathcal{H}_k$.
	Moreover, all critical points of the restriction $E_{\mathcal{N}_k}$
	of $E$ to $\mathcal{N}^{R_k}_k$ are critical points of $E$ with no zero component.
\end{lema}
\begin{proof}
	We show that $\mathcal{N}^{R_k}_k$ is a manifold first. We may write
	$$\mathcal{N}^{R_k}_k = \{(u_1,\cdots, u_{k+1})  \in H_k \mid u_i \neq 0 , I(u_1,\cdots, u_{k+1}) = 0\},$$
	where  $I = I(u_1,\cdots, u_{k+1}): \mathcal{H}_k \rightarrow \mathbb{R}^{k+1}$ is given by	
	\[ \begin{array}{rl}
		I_i(u_1,\cdots, u_{k+1})=& \displaystyle\|u_i\|_i^2 - \int_{B_i}\int_{B_i}\dfrac{F(u_i)f(u_i)u_i}{|x|^{\beta}|x-y|^{\mu}|y|^{\beta}}\, \mathrm{d} x\mathrm{d} y \\&\displaystyle - \frac{1}{2}\sum_{\substack{ i\neq j}}^{k+1}\int_{B_i}\int_{B_j}\dfrac{F(u_j)f(u_i)u_i}{|x|^{\beta}|x-y|^{\mu}|y|^{\beta}}\,  \mathrm{d} x\mathrm{d} y,
	\end{array}\]	for $i = 1, \cdots, k+1$. In order to prove that $\mathcal{N}^{R_k}_k$ is a differentiable manifold in $\mathcal{H}_k$, it suffices to check that the matrix
	\[\begin{aligned}
		N := (N_{ij}) =
		\left (\partial_{u_i}I_j(u_1,\cdots, u_{k+1}), u_i) \right)_{i,j=1,\cdots ,k+1}
	\end{aligned}\]
	is nonsingular at each point $(u_1,\cdots, u_{k+1}) \in \mathcal{N}^{R_k}_k$, since it implies that $0$ is a
	regular value of $I$. In fact, by direct computation, we have
	\[ \begin{array}{cc}
		N_{ij} =\displaystyle - \frac{1}{2} \sum_{\substack{ i\neq j}}^{k+1}\left(\int_{B_i}\int_{B_j}\dfrac{f(u_j(y))u_j(y)f(u_i(x))u_i(x)}{|x|^{\beta}|x-y|^{\mu}|y|^{\beta}}\, \mathrm{d} x\mathrm{d} y  \right)
	\end{array} \]	
	for all $i\neq j$ and $i,j=1,\cdots k+1$, and 
	\[ \scriptsize \begin{array}{rl}
		N_{ii}=& \displaystyle 2\|u_i\|_i^2 -  \int_{B_i}\int_{B_i}\dfrac{f(u_i(y))u_i(y)f(u_i(x))u_i(x)}{|x|^{\beta}|x-y|^{\mu}|y|^{\beta}}\, \mathrm{d} x\mathrm{d} y - \int_{B_i}\int_{B_i}\dfrac{F(u_i(y))f'(u_i(x)u_i^2(x)}{|x|^{\beta}|x-y|^{\mu}|y|^{\beta}}\, \mathrm{d} x\mathrm{d} y \\
		&\displaystyle -\int_{B_i}\int_{B_i}\dfrac{F(u_i(y))f(u_i(x))u_i(x)}{|x|^{\beta}|x-y|^{\mu}|y|^{\beta}}\, \mathrm{d} x\mathrm{d} y \\
		&\displaystyle - \frac{1}{2} \sum_{\substack{ i\neq j}}^{k+1}\left(\int_{B_i}\int_{B_j}\dfrac{F(u_j(y))f(u_i(x))u_i(x)}{|x|^{\beta}|x-y|^{\mu}|y|^{\beta}}\, \mathrm{d} x\mathrm{d} y  +\int_{B_i}\int_{B_j}\dfrac{F(u_j(y))f'(u_i(x))u^2_i(x)}{|x|^{\beta}|x-y|^{\mu}|y|^{\beta}}\, \mathrm{d} x\mathrm{d} y\right).
	\end{array}\] for all $i=1,\cdots, k+1$. As $F,f\geq 0$, then $N_{ij}<0$ for $(u_1,\cdots,u_k)\in \mathcal{N}^{R_k}_k$ and $i\neq j$. Moreover,	\[  \scriptsize
	\begin{array}{cl}
		N_{ii}= & \displaystyle 2\left(\int_{B_i}\int_{B_i}\dfrac{F(u_i(y))f(u_i(x))u_i(x)}{|x|^{\beta}|x-y|^{\mu}|y|^{\beta}}\, \mathrm{d} x\mathrm{d} y \displaystyle +\dfrac{1}{2} \sum_{\substack{ i\neq j}}^{k+1}\int_{B_i}\int_{B_j}\dfrac{F(u_j(y))f(u_i(x))u_i(x)}{|x|^{\beta}|x-y|^{\mu}|y|^{\beta}}\,  \mathrm{d} x\mathrm{d} y\right)\\
		&\displaystyle -  \int_{B_i}\int_{B_i}\dfrac{f(u_i(y))u_i(y)f(u_i(x))u_i(x)}{|x|^{\beta}|x-y|^{\mu}|y|^{\beta}}\, \mathrm{d} x\mathrm{d} y  \displaystyle- \int_{B_i}\int_{B_i}\dfrac{F(u_i(y))f'(u_i(x)u_i^2(x)}{|x|^{\beta}|x-y|^{\mu}|y|^{\beta}}\, \mathrm{d} x\mathrm{d} y\\
		&\displaystyle -\int_{B_i}\int_{B_i}\dfrac{F(u_i(y))f(u_i(x))u_i(x)}{|x|^{\beta}|x-y|^{\mu}|y|^{\beta}}\, \mathrm{d} x\mathrm{d} y \\
		&\displaystyle - \frac{1}{2} \sum_{\substack{ i\neq j}}^{k+1}\left(\int_{B_i}\int_{B_j}\dfrac{F(u_j(y))f(u_i(x))u_i(x)}{|x|^{\beta}|x-y|^{\mu}|y|^{\beta}}\, \mathrm{d} x\mathrm{d} y  +\int_{B_i}\int_{B_j}\dfrac{F(u_j(y))f'(u_i(x))u^2_i(x)}{|x|^{\beta}|x-y|^{\mu}|y|^{\beta}}\, \mathrm{d} x\mathrm{d} y\right)\\
		& =\displaystyle \int_{B_i}\int_{B_i}\dfrac{[F(u_i(y))-f(u_i(y))u_i(y)]f(u_i(x))u_i(x)}{|x|^{\beta}|x-y|^{\mu}|y|^{\beta}}\, \mathrm{d} x\mathrm{d} y\\
		&\displaystyle -\int_{B_i}\int_{B_i}\dfrac{F(u_i(y))u_i(y)f'(u_i(x))u^2_i(x)}{|x|^{\beta}|x-y|^{\mu}|y|^{\beta}}\, \mathrm{d} x\mathrm{d} y\\
		& +\displaystyle\frac{1}{2} \sum_{\substack{ i\neq j}}^{k+1}\int_{B_i}\int_{B_j}\dfrac{F(u_j(y))[f(u_i(x))u_i(x)- f'(u_i(x))u^2_i(x)]}{|x|^{\beta}|x-y|^{\mu}|y|^{\beta}}\, \mathrm{d} x\mathrm{d} y  \\ 
		&\leq
		\displaystyle
		\int_{B_i}\int_{B_i}\dfrac{f(u_i(y))u_i(y)f(u_i(x))u_i(x)}{|x|^{\beta}|x-y|^{\mu}|y|^{\beta}}\, \mathrm{d} x\mathrm{d} y\\ 
		&
		+\displaystyle\frac{1}{2} \sum_{\substack{ i\neq j}}^{k+1}\int_{B_i}\int_{B_j}\dfrac{F(u_j(y))[f(u_i(x))u_i(x)- f'(u_i(x))u^2_i(x)]}{|x|^{\beta}|x-y|^{\mu}|y|^{\beta}}\, \mathrm{d} x\mathrm{d} y  \\
	\end{array} \]
	for $i = 1, \cdots, k+1$. By $(f_3)$, as , $f(t),F(t)>0$, and $f(t)t< f'(t)t^2$, for $t>0$, we have $N_{ii}<0$.	By \cite[Lemma A.3.]{nodal}, we may verify as the proof
	of lemma  that $ \det N \neq 0$  at each point of $\mathcal{N}^{R_k}_k$. So $\mathcal{N}^{R_k}_k$ is a differentiable
	manifold in $\mathcal{H}_k$. Next, we verify that any critical point $(u_1, \cdots , u_{k+1})$ of $E\mid_{\mathcal{N}^{R_k}_k}$
	is a critical point of $E$. Indeed, if $(u_1, \cdots , u_{k+1})$ is a critical point of $E\mid_{\mathcal{N}^{R_k}_k}$
	, then there
	are Lagrange multipliers $\lambda_1, \cdots , \lambda_{k+1}$ such that
	\begin{equation}\label{multiplicadordelagrange2}\begin{aligned}
			\lambda_1I^{'}_1 (u_1, \cdots , u_{k+1}) + \cdots + \lambda_{k+1}I^{'}_{k+1}(u_1, \cdots , u_{k+1}) = E^{'}(u_1, \cdots , u_{k+1}).
	\end{aligned}\end{equation}
	The values of the operator identity (\ref{multiplicadordelagrange2}) at points
	$$(u_1, 0, \cdots , 0),(0, u_2, 0, \cdots , 0), \cdots , (0, \cdots , 0, u_{k+1})$$
	form a system
	
	\[\begin{aligned}
		N \left (\begin{matrix}
			\lambda_1\\
			\lambda_2\\
			\vdots\\
			\lambda_{k+1}
		\end{matrix}\right) = \left (\begin{matrix}
			0\\
			0\\
			\vdots\\
			0
		\end{matrix}\right).
	\end{aligned}\]
	Since the matrix $N$ is nonsingular at each point of $\mathcal{N}^{R_k}_k$, $\lambda_1, \cdots, \lambda_{k+1}$ are all zero and $(u_1, \cdots , u_{k+1})$ is a critical point of $E$.
	Finally, for any $(u_1, \cdots , u_{k+1}) \in {\mathcal{N}}^{R_k}_k$.  Consequently,
	\[\begin{array}{rl} \|u_i\|_i^2& =\displaystyle \dfrac{1}{2}\int_{B_i}\int_{B_i}\dfrac{F(u_i)f(u_i)u_i}{|x|^{\beta}|x-y|^{\mu}|y|^{\beta}}\, \mathrm{d} x\mathrm{d} y \\&\displaystyle + \frac{1}{2}\sum_{\substack{ j=1}}^{k+1}\int_{B_i}\int_{B_j}\dfrac{F(u_j)f(u_i)u_i}{|x|^{\beta}|x-y|^{\mu}|y|^{\beta}}\,  \mathrm{d} x\mathrm{d} y\\
		&=\displaystyle \dfrac{1}{2}\int_{B_i}\int_{B_i}\dfrac{F(u_i)f(u_i)u_i}{|x|^{\beta}|x-y|^{\mu}|y|^{\beta}}\, \mathrm{d} x\mathrm{d} y \\&\displaystyle + \frac{1}{2}\int_{\mathbb{R}^2}\int_{B_i}\dfrac{F(u)f(u_i)u_i}{|x|^{\beta}|x-y|^{\mu}|y|^{\beta}}\,  \mathrm{d} x\mathrm{d} y\\
		&\displaystyle \leq 
		
		C\int_{\mathbb{R}^2}\int_{B_i}\dfrac{F(u)f(u_i)u_i}{|x|^{\beta}|x-y|^{\mu}|y|^{\beta}}\,  \mathrm{d} x\mathrm{d} y.\\
		
	\end{array}\] 
	So,  from  Lemma \ref{WHLS} together with \eqref{growth},  \eqref{growth2}, \eqref{exponencialfinita}  and the Sobolev embedding theorem, using similar arguments to those used to prove \eqref{estimativeF2}, we obtain
	\begin{equation}\label{imersao}\|u_i\|_i^2\leq C(\|u\|\|u_i\|_i)^{\frac{3(4-(\mu+2\beta))}{2}}+C(\|u\|\|u_i\|_i)^{{q}}.\end{equation}
	Since $q,\frac{3(4-(\mu+2\beta))}{2} >2$, there exists a constant $\sigma_i > 0$ such that $\|u_i\|_i \geq \sigma_i $, $i > 0, i = 1, \cdots , k+1$. Thus, critical points of $E$ in $\mathcal{N}^{R_k}_k$
	cannot have any zero component. To conclude, we note that all sequences $(u_1^n, \cdots, u_{k+1}^n)$ converging to an element $(u_1^0, \cdots, u_{k+1}^0)$ of $\mathcal{N}^{R_k}_k$ cannot have any component sequence converging to zero, as $u_i^0$ is bounded away from zero. Thus, the proof is complete.
\end{proof}

	\subsection{Minimizing sequence of $E\mid_{\mathcal{N}_k}$}
	Here, initially, we  present some convergence results for minimizing sequence of $E\mid_{\mathcal{N}_k}$.

	\begin{lema}
		Let $N =2$, $0 < \mu+2\beta < 2$ and $q\in[2, \infty)$. If ${(u^n_1 , \cdots , u^n_{k+1})}$ is a bounded sequence in $L^{p}(B_1)\times \cdots \times L^{p}(B_{k+1}) $ such that
		if ${(u^n_1 ,\cdots , u^n_{k+1})} \rightarrow {(u_1^0 , \cdots , u_{k+1}^0)}$ almost everywhere in $ B_1\times \cdots \times B_{k+1}$ as $n \rightarrow \infty$, then ${(u^n_1 , \cdots , u^n_{k+1})}  \rightharpoonup\ {(u_1 , \cdots , u_{k+1})}$ weakly in $L^q(B_1)\times \cdots \times L^q(B_{k+1}) $.
	\end{lema}
	\begin{lema}\label{lemmaconvef}
		Assume $(f_1)-(f_3)$, and $(f_4)$. Let $\{(u_1^n, \cdots,u_{k+1}^n)\}\subset \mathcal{H}^k$ be a minimizing sequence of $E\mid_{\mathcal{N}_k}$ such that $(u_1^n, \cdots,u_{k+1}^n)\rightharpoonup (u_1, \cdots,u_{k+1})$ weakly in $\mathcal{H}_k$, then
		\[\left[\int_{\mathbb{R}^2} \dfrac{F(u^n_j)}{|x|^{\beta}|x-y|^{\mu}|y|^{\beta}}\,\mathrm{d} y\right] F(u^n_i) \rightarrow\left[\int_{\mathbb{R}^2} \dfrac{F(u_j)}{|x|^{\beta}|x-y|^{\mu}|y|^{\beta}}\,\mathrm{d} y\right] F(u_i) \quad \mbox{in} \quad L^1(\mathbb{R}^2)\]  
		and 
		\[\int_{B_i}\int_{B_j} \dfrac{F(u^n_j)}{|x|^{\beta}|x-y|^{\mu}|y|^{\beta}} f(u^n_i)\phi_i\,\mathrm{d} y\mathrm{d} x \rightarrow\int_{B_i}\int_{B_j} \dfrac{F(u_j)}{|x|^{\beta}|x-y|^{\mu}|y|^{\beta}} f(u_i)\phi_i \,\mathrm{d} y\mathrm{d} x,\] for all $  \phi_i \in C_0^{\infty}(B_i),$ and
		for all  $i,j=1,\cdots,k+1$. 
	\end{lema}
	\begin{proof}
		By Lemma \ref{WHLS}, \eqref{embedding}, \cite[Lemma 2.1]{Miyagaki}, the Lebesgue Dominated Convergence Theorem, and following the arguments in \cite[Lemma 2.4]{alvesetal} without restricting necessarilly to compact sets, once we are considering radially symmetric functions, the proof follows.
	\end{proof}
	
	\begin{lema}\label{Neharibounded}
		Assume $(SG)$, $(f_1)-(f_4)$. There exists a minimizer ${(u_1 , \cdots, u_{k+1})}$  of $E\mid_{\mathcal{N}_k}$ with nonzero components such that satisfies \eqref{probi}. Moreover, $u_i\in C^1(B_{r_i})$, the outward normal derivatives $\frac{\partial u_i}{\partial \nu} \neq 0$  on $ \partial B_i$ and $(-1)^{i+1}u_i>0$ in $B_i$ for all $i=1, \cdots,k+1$.
	\end{lema}
	\begin{proof} Similarly to proof of Lemma \ref{manifold}, we conlude that given $u=(u_1,\cdots, u_{k+1}) \in \mathcal{N}_k$, then $u_i\neq 0$ for all $i=1,\cdots, k+1$. 
		Moreover, as each component of  $(u_1, \cdots, u_{k+1})$ is  bounded away from zero, there holds \begin{equation}\label{bounded}E(u_1, \cdots , u_{k+1}) - \frac{1}{\theta}\sum_{i=1}^{k+1}\partial_{u_i}E(u_1,\cdots, u_{k+1})u_i \geq   \left(\dfrac{\theta -2}{2\theta}\right)\sum_{i=1}^{k+1} \|u_i\|_i^2 \geq \sigma > 0\end{equation} for some $\sigma > 0$. 
		We aslo get that any minimizing sequence ${(u^n_1 ,\cdots , u^n_{k+1})}$ of $E\mid_{ \mathcal{N}_k}$
		is bounded in $\mathcal{H}_k$. 
	We may assume that the minimizing sequence
	${(u^n_1 ,\cdots, u^n_{k+1})}$ weakly converges to an element $(u^0_1, \cdots, u^0_{k+1})$ in $\mathcal{H}_k$. We claim that $u^0_i \neq 0$ for each $i = 1, \cdots, k + 1$. Indeed, since
	$(u^n_1,\cdots , u^n_{k+1}) \in \mathcal{N}_k$,  we have
	$$\begin{array}{rl}\|u_i^n\|_i^2 =&\displaystyle\dfrac{1}{2}\int_{B_i}\int_{B_i}\dfrac{F(u_i^n(y))f(u_i^n(x))u_i^n(x)}{|x|^{\beta}|x-y|^{\mu}|y|^{\beta}}\, \mathrm{d} x\mathrm{d} y \\+&\displaystyle\dfrac{1}{2} \int_{\mathbb{R}^2}\int_{B_i}\dfrac{F(u_j^n(y))f(u_i^n(x))u_i^n(x)}{|x|^{\beta}|x-y|^{\mu}|y|^{\beta}}\, \mathrm{d} x\mathrm{d} y\\ <&\displaystyle
		\int_{\mathbb{R}^2}\int_{B_i}\dfrac{F(u_j^n(y))f(u_i^n(x))u_i^n(x)}{|x|^{\beta}|x-y|^{\mu}|y|^{\beta}}\, \mathrm{d} x\mathrm{d} y.  \end{array}$$
	Now, by Lemma \ref{WHLS}, we obtain
	\begin{equation*} \int_{B_j}\int_{B_{i}}\dfrac{F(u_j^n(y))f(u_i^n(x))u^n_i(x)}{|x|^{\beta}|x-y|^{\mu}|y|^{\beta}}\, \mathrm{d} x\mathrm{d} y \leq 	\left\|F(u_j^n) \right\|_{\frac{4}{4-(\mu+2\beta)}}	\left\|f(u_{i}^n)u_{i}^n \right\|_{\frac{4}{4-(\mu+2\beta)}}. \end{equation*}	
	By  $(SG)$ and $(f_2)$, for any $\varepsilon>0$, $q>2$ and $\alpha>0$, there exists $C=C(\varepsilon,q, \alpha)>0$ such that 
	$$ |f(s)|\leq \varepsilon |s|^{\frac{3(4-(\mu+2\beta))}{2}-1}+ C|s|^{q-1}[e^{\alpha s^2 }-1]\quad \forall s\in \mathbb{R}.$$
	Then, we have
	\begin{equation*}
	\begin{array}{rl}	\left\|f(u_i^n)u_i^n \right\|_{\frac{4}{4-(\mu+2\beta)}}
		\leq& \varepsilon\|u_i^n\|_{2}^{\frac{3(4-(\mu+2\beta))}{2}}\\+&
		\displaystyle C\|u_i^n\|_{\frac{4qt'}{4-(\mu+2\beta)}}^{q}\left\{\int_{B_i}\left[\left(e^{\frac{4\alpha t}{4-(\mu+2\beta)}\|u^n_i\|_i^2 \frac{u_i^n}{\|u_i^n\|_i^2}}-1\right)\right]\,\mathrm{d}x\right\}^{\frac{4-(\mu+ 2\beta)}{4t}},\end{array}
	\end{equation*}
	with $t,t'>1$ such that $\frac{1}{t}+\frac{1}{t'}=1$.  Thus, by \eqref{TM}, for $\alpha$ small enough, we have 
	\[\int_{B_i}\left[\left(e^{\frac{4\alpha t}{4-(\mu+2\beta)}\|u^n_i\|_i^2 \frac{u_i^n}{\|u_i^n\|_i^2}}-1\right)\right]\,\mathrm{d}x< C.\]
	Consequently, from Sobolev embedding 	
	$$\left\|f(u_{i}^n)u_{i}^n \right\|_{\frac{4}{4-(\mu+2\beta)}} \leq \varepsilon\|u_i^n\|_i^{\frac{3(4-(\mu+2\beta))}{2}}+ C\|u_i^n\|_i^{q}.$$
	Analogously, we have 
	$$\left\|F(u_{j}^n) \right\|_{\frac{4}{4-(\mu+2\beta)}} \leq \varepsilon\|u_j^n\|_j^{\frac{(4-(\mu+2\beta))}{2}}+ C\|u_j^n\|_j^{q}.$$ 		
	Then, for all $i=1, \cdots, k+1$, we have
	\begin{equation}\label{normalimitada}\|u_i^n\|_i^2\leq \left[\sum_{j=1}^{k+1}\left(\varepsilon\|u_j^n\|_j^{\frac{3(4-(\mu+2\beta))}{2}}+ C\|u_j^n\|_j^{q}\right)\right]\left(\varepsilon\|u_i^n\|_i^{\frac{3(4-(\mu+2\beta))}{2}}+ C\|u_i^n\|_i^{q}\right).\end{equation} Now, we will show that  $\|u_i^n\|_i$ is bounded from bellow for all $i=1, \cdots, k+1$. Note that if there exist $a_i$ such that $\|u_i^0\|_i\geq a_i$ for $i=1,\cdots, k+1$, we can conclude that  $\|u_i^n\|_i\geq a_i$ for $i=1,\cdots, k+1$.  By \eqref{normalimitada} and \eqref{embedding}, we have
	\[ \begin{array}{rl}
		\|u_i^0\|_i^2 &\leq\liminf_{n \rightarrow + \infty} \|u_i^n\|_i^2 \\ &\leq \displaystyle\lim_{n\rightarrow \infty}\left[\sum_{j=1}^{k+1}\left(\varepsilon\|u_j^n\|_j^{\frac{3(4-(\mu+2\beta))}{2}}+ C\|u_j^n\|_j^{q}\right)\right]\left(\varepsilon\|u_i^n\|_i^{\frac{3(4-(\mu+2\beta))}{2}}+ C\|u_i^n\|_i^{q}\right)\\
		&	\leq \displaystyle\left[\sum_{j=1}^{k+1}\left(\varepsilon\|u_j^0\|_j^{\frac{3(4-(\mu+2\beta))}{2}}+ C\|u_j^0\|_j^{q}\right)\right]\left(\|u_i^0\|_i^{\frac{3(4-(\mu+2\beta))}{2}}+\|u_i^0\|_i^{q} \right).
	\end{array}\]
	Therefore, as $\frac{3(4-(\mu+2\beta))}{2},q
	>2$ there exists a constant $\sigma_i > 0$ such that $\|u_i^0\|_i \geq \sigma_i > 0 $, $ i = 1, \cdots , k+1$.	 Now suppose  that $(u^n_1,\cdots , u^n_{k+1})$ does not strongly converges to $(u^0_1, \cdots, u^0_{k+1})$ in $\mathcal{H}_k$ as
	$n \rightarrow + \infty$. That is, $\|u^0_i\|_i < \liminf_{n \rightarrow + \infty} \|u^n_i\|_i$ for at least one $i \in \{1, \cdots , k + 1\}$. Since each component of $(u^0_1, \cdots, u^0_{k+1})$ is nonzero, by Lemma \ref{multipleneharivetor}, one
	can find $(t^0_1, \cdots, t^0_{k+1}) \in (\mathbb{R}>0)^{k+1}$ and $(t^0_1, \cdots, t^0_{k+1}) \neq (1, \cdots , 1)$ such that
	$(t^0_1u^0_1, \cdots, t^0_{k+1}u^0_{k+1}) \in \mathcal{N}_k$. But, in this case, by  Lemmas \ref{multipleneharivetor} and \ref{lemmaconvef},  we
	derive that
	\[\begin{aligned}
		&\inf_{{(u^n_1 ,\cdots , u^n_{k+1})} \in \mathcal{N}_k}E(u_1, \cdots , u_{k+1}) \leq E(t^0_1u^0_1, \cdots, t^0_{k+1}u^0_{k+1})\\
		&<\liminf_{n \rightarrow + \infty} \left \{\displaystyle\sum_{i=1}^{k+1} \left((t^0_i)^2\dfrac{1}{2}\|u_i^n\|_i^2 - \sum_{j=1}^{k+1}\int_{B_i}\int_{B_j}\dfrac{F(t_i^0u_i^n)F(t_j^0u_j^n)}{|x|^{\beta}|x-y|^{\mu}|y|^{\beta}}\, \mathrm{d} x\mathrm{d} y\right)\right \} \\
		&\leq \liminf_{n \rightarrow + \infty} E(u^n_1, \cdots , u^n_{k+1}) = \inf_{{(u^n_1 ,\cdots , u^n_{k+1})} \in \mathcal{N}_k}E(u_1, \cdots , u_{k+1}).
	\end{aligned}\]	which is a contradiction. Therefore, $(u^n_1,\cdots , u^n_{k+1})$ strongly converges to $(u^0_1, \cdots, u^0_{k+1})$ as
	$n \rightarrow + \infty$ in $\mathcal{H}_k$  and $(u^0_1, \cdots, u^0_{k+1}) \in \mathcal{N}_k$ is a minimizer of $E\mid_{\mathcal{N}_k}$.
	Furthermore, we may check that
	$(w_1, \cdots ,w_{k+1}) := ((|u^0_1|,-|u^0_2|, \cdots,(-1)^{k+1} |u^0_{k+1}|)$
	is also in $\mathcal{N}_k$ and is a minimizer of $E\mid_{\mathcal{N}_k}$. Hence, it is a critical point of
	$E\mid_{\mathcal{N}_k}$. By Lemma \ref{manifold}, it is also a critical point of $E$ and satisfies \eqref{probi}. Then by Hopf’s lemma and the strong maximum principle $\frac{\partial w_i}{\partial \nu} \neq 0$ and $(-1)^{i+1}w_i$ is positive in $B_i$. The assertion follows.
	
\end{proof}

\subsection{Proof of main Theorem \ref{teo1} }
Here, our purpose is to show that a critical point of $E$ on $\mathcal{N}^{R_k}_k$ is also a critical point of $J$. For this matter, first, we will prove some results that will help us to demonstrate this.

\begin{lema}\label{tecnic}
	Assume $(SG)$ and $(f_1)-(f_2)$.	Let $d(R_k)$ be given in \eqref{dk2}, $k \in \mathbb{N}_+$, and $R_k =(r_1, \cdots, r_k)\in \Gamma_k$. Then
	\begin{description}
		\item[(i)] $d(R_k)>0$.
		\item[(ii)] if $r_i- r_{i-1}\rightarrow0$ for some $i=1,\cdots, k$, then $d(R_k) \rightarrow + \infty$. 
		\item[(iii)] if $r_k\rightarrow \infty$, then $d(R_k) \rightarrow + \infty$. 
		
		\item[(iv)] $d$ is continuous in $\Gamma_k$. As a consequence, there exists a $\bar{R}_k\in\Gamma_k$ such that
		$$d(\bar{R}_k) = \inf_{{R}_k\in \Gamma_k}
		d({R}_k).$$
	\end{description}
\end{lema} 
\begin{proof}\textbf{(i):} It follows using the same steps in Lemma \cite[3.1 (ii)]{GuiGuo}.
	
	\textbf{(ii):} Take $u= (u_1^{R_k}, \cdots, u_{k+1}^{R_k})\in \mathcal{N}^{R_k}_k$. Suppose that $r_{i_0}-r_{i_0-1}\rightarrow 0$ for some $i_0\in\{1,\cdots, k\}$, by Lemma \ref{WHLS}, \eqref{growth}, \eqref{growth2}, the Hölder inequality, we get
	we have
	\begin{equation}\label{i0}	
		\begin{array}{rl}\|u^{R_k}_{i_0}\|^2_{i_0}=&\displaystyle \int_{B_{i_0}}\int_{B_{i_0}}\dfrac{F(u^{R_k}_{i_0})f(u^{R_k}_{i_0})u^{R_k}_{i_0}}{|x|^{\beta}|x-y|^{\mu}|y|^{\beta}}\,  \mathrm{d} x\mathrm{d} y + \frac{1}{2}\int_{\mathbb{R}^2}\int_{B_{i_0}}\dfrac{F(u^{R_k}_{j})f(u^{R_k}_{i_0})u^{R_k}_{i_0}}{|x|^{\beta}|x-y|^{\mu}|y|^{\beta}}\,  \mathrm{d} x\mathrm{d} y\\
			\leq&\displaystyle 		
			\int_{\mathbb{R}^2}\int_{B_i}\dfrac{F(u^{R_k}_{j})f(u^{R_k}_{i_0})u^{R_k}_{i_0}}{|x|^{\beta}|x-y|^{\mu}|y|^{\beta}}\,  \mathrm{d} x\mathrm{d} y\\
			\leq&  C\left(\|u_{i_0}^{R_k}\|^{\frac{3(4-(\mu+ 2\beta))}{2}}_{6}+ \|u_{i_0}^{R_k}\|^q_{\frac{8q}{4-(\mu+2\beta)}}\right).\end{array}	\end{equation}
	By Radial Lemma, we obtain
	\begin{equation}
		\label{i01}
		\begin{array}{rl}\|u_{i_0}^{R_k}\|^{\frac{3(4-(\mu+ 2\beta))}{2}}_{6}& \displaystyle\leq C\|u_{i_0}^{R_k}\|_{i_0}^{\frac{3(4-(\mu+ 2\beta))}{2}}\left(\int_{B_{i_0}}  \frac{1}{|x|^3}\mathrm{d} x\right)^{\frac{4-(\mu+ 2\beta)}{4}}\\
			&= C\|u_{i_0}^{R_k}\|_{i_0}^{\frac{3(4-(\mu+ 2\beta))}{2}}\left|\frac{1}{{r}_{i_0}}- \frac{1}{r_{i_{0}-1}}\right|^{\frac{4-(\mu+ 2\beta)}{4}}.\end{array}	\end{equation} 
	Analogously, we get 
	\begin{equation}
		\label{i02}
		\begin{array}{rl}
			\|u_{i_0}^{R_k}\|^2_{\frac{8q}{4-(\mu+2\beta)}} &\leq C \|u_{i_0}^{R_k}\|_{i_0}^{q}|r_{i_0}^{2- \frac{4q}{4-(\mu+ 2\beta) }}- r_{i_0-1}^{2- \frac{4q}{4-(\mu+ 2\beta) }}|^{\frac{4-(\mu+2\beta)}{8}}. 
	\end{array}	\end{equation}
	So, by \eqref{i0}, \eqref{i01},  \eqref{i02}	and Sobolev inequality,  we get
	\[\begin{array}{rl}\|u_{i_0}^{R_k}\|^2_{i_0}& \leq C\|u_{i_0}^{R_k}\|_{i_0}^{\frac{3(4-(\mu+ 2\beta))}{2}}\left|\frac{1}{r_{i_0}}- \frac{1}{r_{i_{0}-1}}\right|^{\frac{4-(\mu+ 2\beta)}{4}} \\&+ C \|u_{i_0}^{R_k}\|_{i_0}^{q}|r_{i_0}^{2- \frac{4q}{4-(\mu+ 2\beta) }}- r_{i_0-1}^{2- \frac{4q}{4-(\mu+ 2\beta) }}|^{\frac{4-(\mu+2\beta)}{8}}.\end{array}\]
	As $\frac{3(4-(\mu +2\beta))}{2},q>2$, $\|u_{i_0}\|_{i_0}\rightarrow \infty$ as $r_{i_0}-r_{i_0-1}\rightarrow0$, which implies that
	$$d(R_k)= E(u_1^{R_k},\cdots, u_k^{R_k}) \geq \left(\frac{1}{2}- \frac{1}{\theta}\right)\|u_{i_0}^{R_k}\|\rightarrow \infty.$$
	Therefore, this item holds.
	
	{\textbf{(iii):} As in \eqref{i0}, we get
		\begin{equation}\label{ik} \scriptsize		
	\begin{array}{l}\|u^{R_k}_{k}\|^2_{k}\leq\displaystyle 		
				\int_{\mathbb{R}^2}\int_{B_i}\dfrac{F(u^{R_k}_{j})f(u^{R_k}_{k})u^{R_k}_{k}}{|x|^{\beta}|x-y|^{\mu}|y|^{\beta}}\,  \mathrm{d} x\mathrm{d} y\\\leq \displaystyle C\left[\sum_{j=1}^{k+1}\left(\|u_{j}^{R_k}\|^{\frac{3(4-(\mu+ 2\beta))}{2}}_{6}+ \|u_{j}^{R_k}\|^q_{\frac{8q}{4-(\mu+2\beta)}}\right)\right] \left(\|u_{k}^{R_k}\|^{\frac{3(4-(\mu+ 2\beta))}{2}}_{6}+ \|u_{k}^{R_k}\|^q_{\frac{8q}{4-(\mu+2\beta)}}\right).\end{array}\end{equation}
		So by Radial Lemma and Sobolev inequality,
		\begin{equation*}\label{ik2}\small
			\begin{array}{rl}\|u^{R_k}_{k}\|^2_{k}&\leq \left(\|u^{R_k}\|^{\frac{3(4-(\mu+ 2\beta))}{2}}+ \|u^{R_k}\|^q\right) \left(\dfrac{\|u_{k}^{R_k}\|_k^{\frac{3(4-(\mu+ 2\beta))}{2}}}{r_k^{\frac{4-(\mu+2\beta)}{4}}}+ \dfrac{\|u_{k}^{R_k}\|_k^q}{r_k^{\left(\frac{4q}{4-(\mu+2\beta)}-2\right)\frac{4-(\mu+2\beta)}{8}}}\right),\end{array}\end{equation*}where $u^{R_k}= \sum_{j=1}^{k+1} u_j^{R_k}$.		Then,  as   $\|u_{k}^{R_k}\|_k$ is bounded from bellow, $ \frac{3(4-(\mu+ 2\beta))}{2},q>2$ and $\frac{4-(\mu+2\beta)}{4},\frac{4q}{4-(\mu+2\beta)}-2>0$. So  either $\|u_{k}^{R_k}\|_k\rightarrow \infty$  or $\|u_{k}^{R_k}\|_k$ is bounded as $r_k\rightarrow \infty$. In this last case, we may conclude that
		\begin{equation*}		
			\begin{array}{rl}1&\leq \left(\|u^{R_k}\|^{\frac{3(4-(\mu+ 2\beta))}{2}}+ \|u^{R_k}\|^q\right) \left(\dfrac{C}{r_k^{\frac{4-(\mu+2\beta)}{4}}}+ \dfrac{C}{r_k^{\left(\frac{4q}{4-(\mu+2\beta)}-2\right)\frac{4-(\mu+2\beta)}{8}}}\right).\end{array}\end{equation*}So  as ${r}_k \rightarrow \infty$,   it follows that  $\|u^{R_k}\|\rightarrow \infty $.
		In both cases, we have
		$$\begin{array}{rl}d(R_k)&=\displaystyle E(u_1^{R_k}, \cdots, u_{k+1}^{R_k})- \dfrac{1}{\theta} \sum_{i=1}^{k+1}E'(u_1^{R_k}, \cdots, u_{k+1}^{R_k})u_k^{R_k}\\&\displaystyle \geq \left(\frac{1}{2}- \frac{1}{\theta}\right)\left(\sum_{i=1}^{k+1}\|u_i^{R_k}\|_i\right) \rightarrow \infty,\end{array}$$
		as $r_k \rightarrow \infty$. Thus, \textbf{(iii)} follows.}
	
	\textbf{(vi):} It follows using the same steps as in
	Lemma 3.1 (ii) in \cite{nodal}.
\end{proof}
\begin{prop}\label{prooftheo1}
	Assume $(SG)$ and $(f_1)-(f_4)$. For each $k\in \mathbb{N}_+$, there exists a radial solution with $k$ nodes $u_k\in \mathcal{N}_k$ of \eqref{problem} such that $J(u_k)= c_k$.\end{prop}
\begin{proof}
	By Lemmas \ref{Neharibounded} and \ref{tecnic}, there exists $\bar{r}_k \in \Gamma_k$ and $(w_1^{\bar{r}_k}, \cdots, w_{k+1}^{\bar{r}_k})\in \mathcal{N}^{\bar{r}_k}_k$ such that $(-1)^{i+1}w_i^{\bar{r}_k}>0$ in $B^{\bar{r}_k}_i$ for $i=1,\cdots, k+1$ and $E(w_1^{\bar{r}_k}, \cdots, w_{k+1}^{\bar{r}_k})= d(\bar{r}_k)=\inf_{r_k\in \Gamma_k}d(r_k).$	 Moreover, by \eqref{jd2}, we have
	$$c_k= d(\bar{r}_k)=J\left(\sum_{i=1}^{k+1}w^{\bar{r}_k}_i\right).$$
	So we need to prove that $u_k= \sum_{i=1}^{k+1}w^{\bar{r}_k}_i $ is a critical point of $J$ in $H_V$, and, consequently,  by the principle of symmetric criticality, a weak solution of \eqref{problem}. So suppose, by contradiction, that there exists a radial symmetric function $\phi\in C_0^{\infty}(\mathbb{R}^N)$ such that
	\begin{equation}\label{minus} J'\left(\sum_{i=1}^{k+1}w^{\bar{r}_k}_i\right)\phi = -2. \end{equation}
	Take $\tilde{t}= (t_1,\cdots, t_{k+1})$ and $y_0=(1,\cdots, 1)\in \mathbb{R}^{k+1}$. We set $g: \mathbb{R}^{k+1}\times \mathbb{R}\rightarrow H_{V}$ as
	\[g(\tilde{t}, \epsilon):= \sum_{i=1}^{k+1} t_i^{1/q}w_i^{\bar{r}_k} + \epsilon\phi.\]
	By Lemma \ref{Neharibounded}, $\sum_{i=1}^{k+1} w_i^{\bar{r}_k} $, changes sign exactly $k$ times, then by continuity of $g$, there exists $0<\tau<1$ small enough  such that for all $(\tilde{t}, \epsilon)\in B_{\tau}(y_0) \times [0,\tau]$, $g(\tilde{t},\epsilon)$ also changes sign $k$ times with $k$ nodes $0< \rho_1(\tilde{t},\epsilon)< \cdots< \rho_k(\tilde{t},\epsilon)< \infty$. Notice that for each $i=1,\cdots, k$, $\rho_i(\tilde{t},\epsilon)$ is continuous about $(\tilde{t},\epsilon)$ in $B_{\tau}(y_0)\times [0,\tau]$, so by continuity and \eqref{minus}, we can take $\tau$ small enough such that 
	\begin{equation}\label{minus2} J\left(g(\tilde{t},\epsilon)\right)\phi < -1 \quad \mbox{for all} \quad (\tilde{t}, \epsilon)\in B_{\tau}(y_0) \times [0,\tau]. \end{equation}	Now consider
	\[\bar{g}(\tilde{t})=  \sum_{i=1}^{k+1} t_i^{\frac{1}{q}}w_i^{\bar{r}_k} + \tau \eta(\tilde{t})\phi,\]
	where $\eta: B_{\tau}(y_0)\rightarrow[0,1]$ is a cut-off fuction such that
	$$\eta(\tilde{t})= \left\{\begin{array}{ll}
		1& \mbox{if} \quad |\tilde{t}-y_0|\leq \frac{\tau}{4},\\
		0& \mbox{if} \quad |\tilde{t}-y_0|\geq \frac{\tau}{2},\\
		\in (0,1)& \mbox{otherwise}.
	\end{array}\right.$$ 
	By definition, $\bar{g}$ is continuous and for each $\tilde{t}\in B_{\tau}(y_0)$, $\bar{g}(\tilde{t})$ has exactly $k$ nodes $0< \rho_1(\tilde{t})< \cdots< \rho_k(\tilde{t})< \infty$, which are continous about $\tilde{t}$. Set $\Omega_j= \{x \in \mathbb{R}^N; |x|\in (\rho_j(\tilde{t}), \rho_{j+1}(\tilde{t})) \}$ and let $H\in C(\overline{B_{\tau/2}(y_0)},\mathbb{R}^{k+1}),$ given by
	\[H(\tilde{t})= \left(\dfrac{1}{qt_1}H_1(\tilde{t}),\cdots, \dfrac{1}{qt_{k+1}}H_{k+1}(\tilde{t})\right) \mbox{with} \quad H_j(\tilde{t}):= \langle J'(\bar{g}(\tilde{t})), \bar{g}(\tilde{t})\mid_{\Omega_j}\rangle.\]
	Now we assert the following.
	\begin{aff}\label{claim}
		There exists $\tilde{t}_0\in B_{\tau/2}(y_0)$ such that $\bar{g}(\tilde{t}_0)\in \mathcal{N}_k$.
	\end{aff}
	By definition of $\mathcal{N}_k$, we need to show that $H(\tilde{t}_0)=0$. To prove this,  we consider a homotopy map $h_s: \overline{B_{\tau/2}(y_0)}\rightarrow \mathbb{R}^{k+1}$, with $s\in[0,1]$, by $$h_s(\tilde{t})= sH(\tilde{t})+(1-s)(y_0 - \tilde{t}).$$ By \eqref{JE2}, we have
	$$ \hat{G}(\tilde{t})	:=E(t_1^{1/q}w^{\bar{r}_k}_1, \cdots, t_{k+1}^{1/q}w^{\bar{r}_k}_{k+1})= J\left(\sum_{i=1}^{k+1}t_i^{1/q}w^{\bar{r}_k}_i\right),$$
	and by Lemma \ref{multipleneharivetor}, $\hat{G}$ attains its unique global maximum at $y_0$ and is strictly concave in $\overline{(\mathbb{R}>0)^{k+1}}$.  So,
	\[\dfrac{\partial\hat{G}}{\partial t_j}(\tilde{t})= \dfrac{1}{q}t_j^{1/q - 1} J'\left(\sum_{i=1}^{k+1}t_i^{\frac{1}{q}}w^{\bar{r}_k}_i\right)w^{\bar{r}_k}_j\]
	and for any $\hat{t} \neq y_0$ in $\overline{(\mathbb{R}>0)^{k+1}}$, the directional derivative of $\tilde{G}$ along $y_0- \tilde{t}$ is positive. Precisely,
	\[0 < \sum^{k+1}_{j=1}\dfrac{\partial\tilde{G}}{\partial t_j}(\tilde{t})(1-t_j)=\sum^{k+1}_{j=1} \dfrac{1}{q}t_j^{1/q - 1}(1-t_j) J'\left(\sum_{i=1}^{k+1}t_i^{\frac{1}{q}}w^{\bar{r}_k}_i\right)w^{\bar{r}_k}_j.\]
	As
	\[\frac{1}{qt_j}H_j(\tilde{t})= \frac{1}{q}t_j^{1/q- 1}J'\left(\sum_{i=1}^{k+1}t_i^{1/q}w^{\bar{r}_k}_i\right)w^{\bar{r}_k}_j \quad \mbox{for} \quad \tilde{t} \in \partial B_{\tau/2}(y_0),\] we conclude that for any $t\in[0,1]$ and $\tilde{t}\in \partial B_{\tau/2}(y_0)$
\[\begin{array}{rl} h_s(\tilde{t})(1-\tilde{t}) =& s\left(\sum^{k+1}_{j=1} \dfrac{1}{q}t_j^{1/q - 1}(1-t_j) J'\left(\sum_{i=1}^{k+1}t_i^{\frac{1}{q}}w^{\bar{r}_k}_i\right)w^{\bar{r}_k}_j\right)\\&\displaystyle+ (1-s)\left( \sum_{j=1}^{k+1}(1-t_j)^2\right)>0.\end{array}\]
	By the homotopy invariance and the Brower degree, it follows that
	\[\begin{array}{rl}\mathrm{deg}(H, B_{\tau/2}(y_0),0)&= \mathrm{deg}(y_0 - id,B_{\tau/2}(y_0),0)\\
		&= \mathrm{deg}(- id,B_{\tau/2}(0),0) = (-1)^{k+1}\neq 0.\end{array}\]
	This guarantees that there exists a $\tilde{t}_0 \in B_{\tau/2}(y_0)$ such that $H(\tilde{t}_0)=0$ and, consequently, $\bar{g}(\tilde{t}_0)\in \mathcal{N}_k$. Thus, we prove this claim.	By \eqref{ck} and Claim \ref{claim}, we have
	\begin{equation}\label{estck}
		J(\bar{g}(\tilde{t}_0))\geq c_k.
	\end{equation}
	However, from \eqref{minus2}, 
	\[\begin{array}{rl}	J(\bar{g}(\tilde{t}_0))&=\displaystyle J\left(\sum_{i=1}^{k+1}\tilde{t}_i^{1/q}w^{\bar{r}_k}_i\right)+  \int_0^1 \left\langle J'\left(\sum_{i=1}^{k+1}\tilde{t}_i^{1/q}w^{\bar{r}_k}_i+  \theta\tau \eta(\tilde{t}_0)\phi\right),\tau \eta(\tilde{t}_0)\phi\right\rangle \ud \theta\\
		&\leq\displaystyle J\left(\sum_{i=1}^{k+1}\tilde{t}_i^{1/q}w^{\bar{r}_k}_i\right)-\tau \eta(\tilde{t}_0).\end{array}\]
	For $|\tilde{t}_0 -y_0|< \tau/2$, then $\eta(\tilde{t}_0)>0$, so 
	\[	J(\bar{g}(\tilde{t}_0))\leq J\left(\sum_{i=1}^{k+1}\tilde{t}_i^{1/q}w^{\bar{r}_k}_i\right)\leq J\left(\sum_{i=1}^{k+1}w^{\bar{r}_k}_i\right)= c_k. \]
	For $|\tilde{t}_0 -y_0|\geq \tau/2$, then, by uniqueness of a global maximum point at $y_0=(1,\cdots,1)$ and $\tilde{t}_0\neq y_0$, we obtain 
	\[	J(\bar{g}(\tilde{t}_0))= J\left(\sum_{i=1}^{k+1}\tilde{t}_i^{1/q}w^{\bar{r}_k}_i\right)< J\left(\sum_{i=1}^{k+1}w^{\bar{r}_k}_i\right)= c_k, \]
	which is a contradiction with \eqref{estck}. Thus, $u_k = \sum_{i=1}^{k+1}w^{\bar{r}_k}_i$ is a radial solution \eqref{problem} and changes sign $k$ times. Furthermore, $u_k \in \mathcal{N}_k$ and $J(u_k)= c_k$, which completes the proof. \end{proof}
\section{Critical case}\label{criticalcase} Here, we assume the nonlinearity has critical exponential growth. In this context, we can show almost all results obtained in Section \ref{subcriticalcase} by replacing $(SG)$ and $(f_4)$ with $(CG)$ and $(f_5)$, respectively. However, to prove a result similar to Lemma \ref{Neharibounded}, it is crucial to estimate the minimax level of the functional $E$ to ensure that the obtained solution does not have any zero components. In the subcritical case, this step can be neglected, as we can take $\alpha > 0$ small enough in \eqref{growth}, which is not possible in the critical case. Therefore, before proving the corresponding result to Lemma \ref{Neharibounded}, we need to estimate the minimax level of $E$ to obtain a suitable upper bound for the norms of each component of a solution of \eqref{probi}.
\begin{prop}\label{levelll}
	Let $c_k$ be given as in \eqref{ck}, then 
	\begin{equation}
		c_k< \displaystyle\dfrac{(4-(\mu+2\beta))(\theta-2)}{8\theta}.
	\end{equation}
\end{prop}
\begin{proof}
	To begin with, from the continuos Sobolev embbedings for  $p> 4 - (\mu +2\beta)$, there exists a constant $C>0$ such that $||u||\geq C ||u||_p$, for all $u\in H_V \setminus \{0\}$. Thus, we can define the following real values
	$$
	S_p(u) = \dfrac{||u||}{||u||_p} \textrm{ \ \ and \ \ } S_q = \inf\limits_{u\in H_V \setminus \{0\}} S_p(u)>0.
	$$
	Now, for each $i=1,\cdots, k+1$, we fix $\psi_i$ being a function in $C^{\infty}_0(B_i)\setminus\{0\}$. Using a similar result to Lemma \ref{multipleneharivetor}, we know that there is $(\overline{t}_1, \cdots, \overline{t}_{k+1})\in (\mathbb{R}>0)^{k+1}$ such that $(\overline{t}_{1}^{\frac{1}{q}}\psi_1, \cdots, \overline{t}_{k+1}^{\frac{1}{q}}\psi_{k+1})\in \mathcal{N}_{k}^{R_k} $.
	Hence, by $(f_5)$, we verify that 
	\[\begin{array}{rl}
		c_k \leq & E(\overline{t}_{1}^{\frac{1}{q}}\psi_1, \cdots, \overline{t}_{k+1}^{\frac{1}{q}}\psi_{k+1})\\ 
		\leq& \displaystyle\max_{(t_1,\cdots,t_{k+1})\in [0,\infty)\times\cdots\times [0,\infty)}E((t_1\psi_1,\cdots, t_{k+1}\psi_{k+1}))\\
		\leq&\displaystyle\max_{(t_1,\cdots,t_{k+1})\in [0,\infty)\times\cdots\times [0,\infty)}\sum_{i=1}^{k+1}\left(\dfrac{t_i^2}{2}\|\psi_i\|_i^2 -\dfrac{t_i^{2p}C_p}{2}\int_{B_i}\int_{\mathbb{R}^2}\dfrac{|\psi_j|^p|\psi_i|^p}{|x|^{\beta}|x-y|^{\mu}|y|^{\beta}}\mathrm{d}y\, \mathrm{d} x\right)\\
		\leq& \dfrac{(p-1)\sum_{i=1}^{k+1}S_{p}(\psi_i)^{\frac{2p}{p-1}}}{2p^{\frac{p}{p-1}}C_p^{\frac{2}{p-1}}}.
	\end{array}\]
	Taking the infimum over all $\psi_i$ and choosing	\[\label{Cp}
	C_{p}>\frac{\left[\dfrac{4\theta(p-1)(k+1) S_{p}^{\frac{2p}{p-1}}}{[4-(\mu+2\beta)](\theta-2)}\right]^{(p-1)/2}}{p^{\frac{p}{2}}},
	\] we have
	\[c_k <\dfrac{(4-(\mu+2\beta))(\theta-2)}{8\theta}. \]
\end{proof}

\begin{lema}\label{Neharibounded2}
	There exists a minimizer ${(u_1 , \cdots, u_{k+1})}$  of $E\mid_{\mathcal{N}_k}$ with nonzero components such that satisfies \eqref{probi}. Moreover, $u_i\in C^1(B_{r_i})$, the outward normal derivatives $\frac{\partial u_i}{\partial \nu} \neq 0$  on $ \partial B_i$ and $(-1)^{i+1}u_i>0$ in $B_i$ for all $i=1, \cdots,k+1$.
\end{lema}
\begin{proof} Similarly to proof of Lemma \ref{Neharibounded}, we conlude that any minimizing sequence ${(u^n_1 ,\cdots , u^n_{k+1})}$ of $E\mid_{ \mathcal{N}_k}$
	is bounded in $\mathcal{H}_k$. Consequently, any minimizing sequence
	${(u^n_1 ,\cdots, u^n_{k+1})}$ weakly converges to an element $(u^0_1, \cdots, u^0_{k+1})$ in $\mathcal{H}_k$. To prove that $u^0_i \neq 0$ for each $i = 1, \cdots, k + 1$, notice that
	$(u^n_1,\cdots , u^n_{k+1}) \in \mathcal{N}_k$,  we have
	$$\begin{array}{rl}\|u_i^n\|_i^2 \leq \displaystyle
		C\int_{\mathbb{R}^2}\int_{B_i}\dfrac{F(u_j^n(y))f(u_i^n(x))u_i^n(x)}{|x|^{\beta}|x-y|^{\mu}|y|^{\beta}}\, \mathrm{d} x\mathrm{d} y,  \end{array}$$
	and, by Lemma \ref{WHLS}, we obtain
	\begin{equation*} \int_{B_j}\int_{B_{i}}\dfrac{F(u_j^n(y))f(u_i^n(x))u^n_i(x)}{|x|^{\beta}|x-y|^{\mu}|y|^{\beta}}\, \mathrm{d} x\mathrm{d} y \leq 	\left\|F(u_j^n) \right\|_{\frac{4}{4-(\mu+2\beta)}}	\left\|f(u_{i}^n)u_{i}^n \right\|_{\frac{4}{4-(\mu+2\beta)}} \end{equation*}	
	By  $(CG)$ and $(f_2)$, for any $\varepsilon>0$, $q>1$ and $\alpha>1$, there exists $C=C(\varepsilon,q, \alpha)>0$ such that 
	$$ |f(s)|\leq \varepsilon |s|^{\frac{3(4-(\mu+2\beta))}{2}-1}+ C|s|^{q-1}[e^{\alpha4\pi s^2 }-1]\quad \forall s\in \mathbb{R}.$$
	Then, we have
	\begin{equation*}\begin{array}{l}
		\left\|f(u_i^n)u_i^n \right\|_{\frac{4}{4-(\mu+2\beta)}}
		\leq\displaystyle \varepsilon\|u_i^n\|_{2}^{\frac{3(4-(\mu+2\beta))}{2}}\\\displaystyle +C\|u_i^n\|_{\frac{4qt'}{4-(\mu+2\beta)}}^{q}\left\{\int_{B_i}\left[\left(e^{\frac{4\alpha t}{4-(\mu+2\beta)}\|u^n_i\|_i^24\pi \frac{u_i^n}{\|u_i^n\|_i^2}}-1\right)\right]\,\mathrm{d}x\right\}^{\frac{4-(\mu+ 2\beta)}{4t}},\end{array}
	\end{equation*}
	with $t,t'>1$ such that $\frac{1}{t}+\frac{1}{t'}=1$. Notice that, by Proposition \ref{levelll} and \eqref{bounded}, we obtain
	$$\lim_{n\rightarrow 0}\sum_{i=1}^{k+1}\|u_{i}^n\|_i^2= c_k< \frac{4-(\mu+2\beta)}{4}. $$ 
	So, there exist $\delta>0$ small enough and $n_0>0$ such that
	\begin{equation}
		\label{normbounded}
		\sum_{i=1}^{k+1}	\|u_{i}^n\|_i^2\leq \frac{4-(\mu+2\beta)}{4}(1-\delta) \quad \forall n>n_0.
	\end{equation} Thus, choosing $\alpha,t>1$ such that $1< \alpha t < \dfrac{1}{1-\delta}$, from \eqref{normbounded}, we get
	\[\dfrac{4\alpha t}{4-(\mu+\beta)}\|u_i^n\|_i^2<1 \quad \forall i=1,\cdots, k+1 \quad \mbox{and} \quad \forall n>n_0.\] Thus, by \eqref{TM}, we have 
	\[\int_{B_i}\left[\left(e^{\frac{4\alpha t}{4-(\mu+2\beta)}\|u^n_i\|_i^24\pi \frac{u_i^n}{\|u_i^n\|_i^2}}-1\right)\right]\,\mathrm{d}x< C.\]
	Now, in order to conclude the result, we follow exactly the same steps used in the proof of Lemma \ref{Neharibounded}.
	
\end{proof}
\subsection{Proof of main Theorem \ref{teo2} }
At first, notice that Lemma \ref{tecnic} is also avaliable when we assume $(CG)$.

\begin{prop}\label{prooftheo2}
	Assume $(CG)$ and $(f_1)-(f_3)$ and $(f_5)$. For each $k\in \mathbb{N}_+$, there exists a radial solution with $k$ nodes $u_k\in \mathcal{N}_k$ of \eqref{problem} such that $J(u_k)= c_k$.\end{prop}

\begin{proof} The proof follows the same steps in a proof of Proposition \ref{prooftheo1}.\end{proof}

\textbf{Acknowledgements}

The first author was supported  by  FAPESP/Brazil Proc. 2023/18443-8. The second author was  supported by FAPESP/Bazil Proc. 2023/05445-2. The third author was partially supported  by CNPq/Brazil Proc 303256/2022-1 and FAPESP/Brazil Proc 2022/16407-1. The forth author was supported by State University of Santa Cruz/Brazil (PROPP 073.6766.2019.0020740-89).

This paper was written while Eudes M. Barboza was on
Posdoctoral stage in the Department of Mathematics of the Federal University of
São Carlos, whose hospitality he gratefully acknowledges.

\vspace{1cm}
\noindent\textsc{EUDES M. BARBOZA}\\
Departamento de Matem\'atica,\\ 
Universidade Federal Rural de Pernambuco - UFRPE\\
50740-560, Recife, Pernambuco, Brazil\\
\noindent\texttt{eudes.barboza@ufrpe.br}.\\

\noindent\textsc{EDUARDO DE S. BÖER}\\
Instituto de Ci\^encias Matem\'aticas e de Computa\c c\~ao\\
Universidade de S\~ao Paulo -- USP\\
13566-590, Centro, S\~ao Carlos - SP, Brazil\\
\noindent\texttt{eduardoboer@usp.br}. \\

\noindent\textsc{OLÍMPIO H. MIYAGAKI}\\
Departamento de Matemática\\
Universidade Federal de S\~ao Carlos - UFSCar\\
13565-905,
São Carlos, São Paulo,
Brazil\\
\noindent\texttt{olimpio@ufscar.br}. \\

\noindent\textsc{CLAUDIA R. SANTANA}\\
Departamento de Ci\^encias Exatas e Tecnol\'ogicas\\
Universidade Estadual de Santa Cruz - UESC\\
45662-900,		
		Ilh\'eus, Bahia, Brazil\\
\noindent\texttt{santana@uesc.br}. 
\end{document}